\documentclass[final]{siamltex}

\usepackage{epsfig,appendix}
\usepackage[english]{babel}
\usepackage{amssymb,amsmath,amsfonts}
\usepackage{mathrsfs,empheq}
\usepackage{mdwlist}
\usepackage{enumerate,url,enumitem,multirow,multicol}
\usepackage[usenames,dvipsnames,svgnames]{xcolor}

\usepackage{algorithm,extarrows}
\usepackage{algcompatible}
\algnewcommand\algorithmicreturn{\textbf{return}}
\algnewcommand\RETURN{\State \algorithmicreturn}%
\usepackage[outercaption]{sidecap}

\usepackage{tikz}
\usepackage{bm}

\usepackage{pgfplots}
\pgfplotsset{
  tick label style={font=\footnotesize},
  label style={font=\footnotesize},
  legend style={font=\footnotesize}
}

\usepackage{booktabs}
\usepackage{caption,subcaption}

\usepackage{todonotes,bbding,stackengine}

\usepackage[normalem]{ulem} 

\usepackage[colorlinks=true, allcolors=blue]{hyperref} 
\usepackage{cleveref}  

\newcommand{\bz}{z}

\newcommand{\rsf}{\mathsf{r}}
\newcommand{\ksf}{\mathsf{k}}

\newcommand{\ran}{\mathrm{ran}}
\newcommand{\ransf}{\mathrm{ran}}
\newcommand{\kersf}{\mathrm{ker}}
\newcommand{\intsf}{\mathrm{int}}
\newcommand{\srisf}{\mathrm{sri}}

\newcommand{\grasf}{\mathrm{gra}}

\newcommand{\Fixsf}{\mathrm{Fix}}

\newcommand{\bw}{w}

\newcommand{\R}{\mathbb{R}}
\newcommand{\N}{\mathbb{N}}
\newcommand{\cI}{\mathcal{I}}

\newcommand{\cS}{\mathcal{S}}

\newcommand{\cN}{\mathcal{N}}

\newcommand{\cT}{\mathcal{T}}

\newcommand{\cA}{\mathcal{A}}
\newcommand{\cU}{\mathcal{U}}

\newcommand{\cP}{\mathcal{P}}

\newcommand{\cH}{\mathcal{H}}

\newcommand{\cC}{\mathcal{C}}

\newcommand{\cY}{\mathcal{Y}}

\newcommand{\cQ}{\mathcal{Q}}

\newcommand{\cK}{\mathcal{K}}
\newcommand{\cB}{\mathcal{B}}

\newcommand{\cTtilde}{\tilde{\mathcal{T}}}

\newcommand{\zersf}{\mathrm{zer}}

\newcommand{\sri}{\mathrm{sri}}

\newcommand{\domsf}{\mathrm{dom}}

\newcommand{\be}{\begin{equation}}
\newcommand{\ee}{\end{equation}}

\newcommand{\weak}{\rightharpoonup}
\newcommand{\infimal}{\triangleright}

\newcommand{\mcup}{\mathbin{\raisebox{-0.20ex} {\scalebox{1.3}{$\cup$}}}} 
 
\newcommand{\mcap}{\mathbin{\raisebox{-0.05ex} {\scalebox{1.3}{$\cap$}}}}

\DeclareMathOperator{\Arg}{Arg}         
\newtheorem{assumption}{Assumption}
\newtheorem{example}{Example}
\newtheorem{remark}{Remark}
\newtheorem{fact}{Fact}
\newif\ifcompilePGFfigs
\compilePGFfigstrue

\graphicspath{{figs/}}

\title{On degenerate preconditioned proximal point methods under restricted monotonicity
}

\author{Feng Xue\thanks{National key laboratory, Beijing,  China (\tt{fxue@link.cuhk.edu.hk}).}
\and Hui Zhang\textsuperscript{\Envelope}\thanks{Corresponding author. Department of Mathematics, National University of Defense Technology, Changsha, Hunan 410073, China (\tt{h.zhang1984@163.com}).} } 

\hypersetup{pdfauthor={Feng Xue},
    pdftitle={Restricted monotonicity in degenerate preconditioned resolvent},
    colorlinks=true,
    linkcolor=blue,
    citecolor=red,
    urlcolor=blue,
}

\date{\today}

\begin{document}

\maketitle

\begin{abstract}
This work investigates the fundamental properties of the degenerate preconditioned resolvent under restricted monotonicity. We extend key notions of non-expansiveness and demiclosedness to the degenerate case. By deriving an explicit characterization of the solution set of the resolvent, we establish several necessary and sufficient conditions for both the well-posedness of the degenerate resolvent and the weak convergence of the associated degenerate proximal point method---considered within either the range space of the preconditioner or the entire Hilbert space. These results provide new insights into the behavior of various operator splitting algorithms, particularly within the range space of the preconditioner. Numerical examples  extend the augmented Lagrangian method and Douglas--Rachford splitting algorithm to non-convex settings under restricted maximal monotonicity.
\end{abstract}




\section{Introduction}  \label{sec_intro}
\subsection{Motivation} \label{sec_metric}
In nonlinear analysis and optimization, a generic and central task is to obtain numerical solutions to the inclusion problem
\be \label{p1}
\textrm{find } x\in \cH, \quad \textrm{s.t. }
0 \in \cA x,
\ee
where $\cH$ denotes a real Hilbert space and $\cA:\cH\mapsto 2^\cH$ is a set-valued maximally monotone operator. The classical proximal point method (PPM), dating back to 1970s \cite{martinet,rtr_1976}, becomes a standard solver of  \eqref{p1}, and its close connections to the augmented Lagrangian method (ALM), particularly from the dual viewpoint, have been thoroughly studied in \cite{rtr_1976_2,rtr_2023,marchi_thesis}. 

When the PPM is equipped with a  linear, positive definite preconditioner $\cQ$, it becomes:
\be \label{T}
x^{k+1} = \cT x^k, \quad \textrm{where\ } \cT :=(\cA+ \cQ)^{-1} \cQ.
\ee
Here, $\cQ$ is sometimes called metric \cite{fxue_rima} and correspondingly, $\cT$ is termed as {\it preconditioned/metric resolvent} \cite{fxue_rima,latafat_2024}. The preconditioned PPM \eqref{T} on its own is a conceptual framework, which was shown in \cite{hbs_siam_2012,hbs_siam_2012_2,hbs_jmiv_2017,bredies_2017,fxue_drs} to encompass a large variety of first-order operator splitting algorithms, such as alternating direction methods of multipliers (ADMM) \cite{boyd_admm} and Chambolle--Pock algorithm \cite{cp_2011}. The PPM \eqref{T} also serves as a basic prototype, upon which many extensions of PPM, such as asymmetric forward-backward-adjoint splitting \cite{latafat_2017,latafat_chapter}, preconditioned forward-backward splitting (FBS) \cite{ecc,fxue_coam} and TriPD \cite{latafat_ieee}, have been proposed to cover Condat--V\~{u}  \cite{condat_2013,vu_2013}, and more complex primal-dual algorithms \cite{arias_2011,bot_2014,bot_2019,ywt_2017,plc_vu}. This operator $\cT$ is a special case of {\it warped resolvent} or {\it $\cQ$-resolvent of $\cA$} \cite[Definition 1.1]{plc_warped}, which allows $\cQ$ to be non-linear. The preconditioned PPM is also closely related to the Bregman monotone proximal point algorithm \cite{hhb_resolvent,plc_bregman,hhb_bregman}, which replaces the standard norm by Bregman distance associated with a Legendre function, and $\cT$ is thus also called {\it $D$-resolvent} of $\cA$ in this context.

Very recently, the linear preconditioner $\cQ$ has been extended to be positive {\it semi}-definite, rather than strictly positive definite, and thus termed {\it degenerate} in \cite{bredies_2017,bredies_ppa,fxue_drs,latafat_2017}.  This degenerate PPM has been studied in a systematic way in infinite dimensional setting in  \cite{bredies_ppa}.
The degeneracy phenomenon prevails in many operator splitting algorithms. For instance, the celebrated Douglas--Rachford splitting (DRS), for solving $\min_{u\in\cU} f(u)+g(u)$, reads as
\be  \label{drs}
\left\lfloor \begin{array}{lll}
u^{k+1} & = &    \arg \min_u   g(u) +
\frac{1}{2\tau } \big\| u -  \bz^k  \big\|^2, \\
 \bw^{k+1} & = & \arg \min_\bw  f (\bw) + \frac{1}{2 \tau} \big\|  \bw -   (2u^{k+1} -  \bz^{k} ) \big\|^2 , \\
\bz^{k+1} & =  & \bz^k +(\bw^{k+1} - u^{k+1}).
 \end{array} \right.
\ee 
It exactly fits the form of \eqref{T}  according to
\be  \label{drs_ppa}
\cH = \cU^3,\ x^k =  \begin{bmatrix}
u^{k} \\ w^{k} \\ z^{k} \end{bmatrix},\ 
 \cA =   \begin{bmatrix}
\tau \partial g & I  & -I \\
-I  &\tau \partial f &  I \\
I   & -I  &  0
\end{bmatrix},\ 
\cQ =    \begin{bmatrix}
  0  &  0 &  0 \\   0 &  0    &  0 \\  0 &  0  &  I
\end{bmatrix},
\ee 
where the preconditioner $\cQ$ is clearly degenerate. Here, we adopt a different (but essentially equivalent) scheme \eqref{drs} from \cite[Eq.(1.6)]{bredies_ppa} for two reasons: (i) \eqref{drs} explicitly preserves both solution trajectories $\{(u^k, w^k)\}_{k\in\N}$ and the shadow sequence $\{z^k\}_{k\in\N}$; (ii) the degenerate preconditioner $\cQ$ in \eqref{drs_ppa} is easier to analyze than that  in  \cite[Eq.(1.5)]{bredies_ppa}. 

The degeneracy of $\cQ$ makes the iterative behaviours of \eqref{T} more subtle. A normal iterative process should require the operator $T$ to be well-defined.  Here the meaning of {\it well-definedness} is characterized as follows: 
\begin{itemize}
\item $\cT$ has to be of {\it full domain}, such that \eqref{T} admits arbitrary input $x^k\in \cH$ during the iterations;
\item $\cT$ has to be {\it single-valued}, such that $x^{k+1}$ can be uniquely determined by the current iterate $x^k$.
\end{itemize}

Unfortunately, the (global) maximal monotonicity of $\cA$ is not sufficient to guarantee the well-definedness of \eqref{T} in the degenerate case, as mentioned in \cite[Remark 2.2]{bredies_ppa}.
This motivates us to find the sufficient conditions on $\cA$ and $\cQ$ for the well-definedness of $\cT$ and the convergence of $\{x^k\}_{k\in\N}$. Equipped with our developed  results, we will be able to generalize the applications of DRS and ALM to non-convex settings.

\subsection{Related works}
\label{sec_related}
For the general warped resolvent, the well-definedness of $\cT$ can be ensured by $\ransf \cQ \subseteq \ransf(\cA+\cQ)$ and disjoint injectivity of $\cA+\cQ$ (see  \cite[Proposition 3.8]{plc_warped} and \cite[Proposition 2.1]{arias_infimal} for instance). For the degenerate case, the well-definedness of $\cT$ was simply put as a basic assumption of the so-called `{\it admissible preconditioner}' \cite[Definition 2.1]{bredies_ppa}, which was subsequently used in \cite{gautam} for studying warped Yosida regularization.  In a more recent work \cite{latafat_2024}, the authors also simply assumed $\cT$ to be of full domain (cf. \cite[Assumption 1]{latafat_2024}) and single-valued continuous (cf. \cite[Theorems 2.3 and 2.6]{latafat_2024}). 
\cite[Proposition 2.4]{naldi_thesis} further characterized the admissibility of $\cQ$ for $\cA$ in the same spirit as \cite{arias_infimal,plc_warped}, where the disjoint injectivity of $\cA+\cQ$ was refined to $\ransf\cQ$.  Besides, a number of sufficient conditions for well-definedness of $\cT$ are listed in \cite[Proposition 3.9]{plc_warped} such as uniform, strict or strong monotonicity. Though these conditions also apply for the degenerate case, they are too strong and hard to verify in practice. More importantly, they fail to associate the maximal monotonicity of $\cA$ with degeneracy of $\cQ$. The degenerate case is not particularly discussed in \cite{arias_infimal,plc_warped,hhb_resolvent,plc_bregman,hhb_bregman}. This motivates our focus to find mild conditions for the degenerate case.

The convergence issue of degenerate version of \eqref{T} has been first tackled in \cite{condat_2013,latafat_2017} for specific algorithms (corresponding to particular $\cA$ and $\cQ$) in finite dimensional setting. However, the continuity of the proximity operator---as a main assumption proposed therein (cf. \cite[Theorem 3.3]{condat_2013} and  \cite[Theorem 3.4]{latafat_2017})---is generally not satisfied in weak topology of infinite dimensional space. For infinite dimensional case, \cite{bredies_2017} discussed the weak convergence for a class of preconditioned ADMM, by reformulating it into a degenerate preconditioned PPM, and relying on the demiclosedness (i.e., weak-strong closedness) of $\cI -\cT$ under several additional assumptions (e.g., compactness and weak closedness in \cite[Lemma 3.4]{bredies_2017}). This analysis was based on the specific structure of $\cA$ and $\cQ$ (corresponding to ADMM), rather than a general treatment of \eqref{T}.  For the general degenerate setting of \eqref{T}, \cite{bredies_ppa} performed a systematic study, while extending the discussion of \cite{latafat_2017} to infinite dimensional settings.

Recently, this degenerate  PPM framework of \cite{bredies_ppa} has been  specialized to many following-up variants, e.g., PPM with adaptive stepsize \cite{lorenz_degenerate}, fast preconditioned ADMM algorithm \cite{sdf_fast}, graph-based algorithms \cite{bredies_lift,compoy_2024}, new extensions of Davis--Yin algorithm \cite{compoy_2024,naldi_2025}. This method was also extended to degenerate FBS \cite{fxue_coam,ecc,naldi_2025} and inexact hybrid proximal extragradient framework \cite{alves_2024}.

\subsection{Contributions and organization}
This work aims to provide a fundamental analysis and understanding of the degenerate preconditioned resolvent and its associated iterative scheme \eqref{T}. 

We first develop an explicit expression of $\cT x$ based on subspace decomposition in Lemma \ref{l_X}, from which naturally follows a notion of restricted monotonicity (cf. Lemma \ref{l_mon_eq} and Definition \ref{def_mon}). This concept connects the monotonicity of $\cA$ with the degeneracy of $\cQ$, and thus, is much milder than the commonly used (global) monotonicity. The operator $\cT$ enjoys several nice properties under restricted monotonicity, such as single-valuedness of $\cP_{\ransf\cQ} \circ \cT$ (Lemma \ref{l_single}) and $\cQ$-firm nonexpansiveness of $\cT$ (Fact \ref{f_fne}). In particular, 
 Lemma \ref{l_demi_T} establishes the $\cQ$-demiclosedness of $\cI-\cP_{\ransf \cQ} \circ \cT|_{\ransf\cQ}$, which obviates our need for a rather demanding assumption of demiclosedness of $\cI-\cT$ (see many additional conditions in \cite[Lemma 3.4]{bredies_2017}).
 
Regarding the full domain of $\cT$, Theorem \ref{t_full} presents a sufficient and necessary condition based on restricted Minty's theorem  (cf. Lemma \ref{l_minty}). This is a clear characterization of `{\it admissible preconditioner}' proposed in \cite{bredies_ppa} and adopted by other related works \cite{gautam,latafat_2017,latafat_2024}. Corollary \ref{c_full} further provides a sufficient condition for the admissibility.
Sect. \ref{sec_max} discusses the maximality of restricted monotonicity and its related issues. Finally, our results are verified by several counterexamples in Sect. \ref{sec_eg}.

For the single-valuedness of $\cT$, Theorem \ref{t_single} presents several equivalent conditions (shown in Proposition \ref{p_single}), which weakens the disjoint injectivity of $\cA+\cQ$ proposed in \cite{bau_review,plc_warped,arias_infimal} to that on $\ransf\cQ$ with respect to (w.r.t.) $\kersf \cQ$. This improvement relies to a large extent on the fact that $\cP_{\ransf\cQ} \circ \cT$ is already single-valued (cf. Lemma \ref{l_single}).

Theorem \ref{t_con_ran} establishes the weak convergence of $\cP_{\ransf\cQ} x^k \weak \cP_{\ransf\cQ} x^\star$ for some $x^\star\in \zersf\cA$, relying on the $\cQ$-demiclosedness of $\cI-\cP_{\ransf\cQ} \circ \cT|_{\ransf\cQ}$. This extends \cite[Theorem 3.4]{latafat_2017} from finite dimensional setting to general Hilbert spaces.  More importantly, Theorem \ref{t_con_ran} does not need to assume the single-valuedness (and hence well-definedness) of $\cT$ as \cite[Theorem 3.4]{latafat_2017}.  Notably, the operator $\cA$ is not assumed to be maximally monotone for this result. Many technical issues are discussed in Sect. \ref{sec_diss}.  We further show in Theorem  \ref{t_con} the whole convergence of $\{x^k\}_{k\in\N}$ under additional conditions. The comparisons with existing works are also discussed in details. 

Instead of our restriction of the PPM within $\ransf\cQ$,  \cite{bredies_ppa,naldi_thesis} directly perform the dimensionality reduction in parallel.  Sect. \ref{sec_reduced} develops several further results along this line, and establishes its connections to our approach. 

The numerical examples in Sect. \ref{sec_app} demonstrate the convergence of the reduced forms of ALM and DRS under restricted maximal monotonicity, even if the objective function is non-convex. This shows that our results broaden the application range of the first-order splitting algorithms.

\subsection{Notations, assumptions and preliminaries}
We use standard notations and concepts in convex analysis and variational analysis from \cite{rtr_book,rtr_book_2,plc_book,beck_book}. 
The adjoint of the linear operator $\cC$ is denoted by $\cC^\top$. The class of proper, lower semi-continuous  and convex functions is denoted by $\Gamma_0(\cH)$. The strong and weak convergences are denoted by $\rightarrow$ and $\rightharpoonup$, respectively. Let $\cT|_C$ denote the operator $\cT$ with its domain being restricted to a subset $C\subset\cH$, and let $\cP_C$ denote the orthogonal projection operator onto a subset $C\subset\cH$. The geometric objects, such as convex cone and (strong) relative interior, are described in  \cite[Sect. 6]{plc_book}. The $\cQ$-based semi-inner product and norm are defined as $\langle \cdot | \cdot \rangle_\cQ = \langle \cQ \cdot |\cdot \rangle$ and  $\|\cdot\|_\cQ = \sqrt{ \langle \cdot|\cdot \rangle_\cQ}$. We use the subscripts of $\rsf$ and $\ksf$ to denote the projections onto $\ransf\cQ$ and $\kersf\cQ$: $x_\rsf := \cP_{\ransf\cQ}  x$ and $x_\ksf := \cP_{\kersf\cQ}  x$. 
 $J_\cA$ denotes a resolvent of a (not necessarily maximally) monotone operator $\cA$. The parallel composition of $\cA$ by  linear operator $\cC$ is denoted by $\cC\infimal \cA = (\cC \circ \cA^{-1} \circ \cC^\top)^{-1}$ \cite[Definition 25.39]{plc_book}.

Throughout the paper, we make the following assumption on $\cQ$.

\begin{assumption} \label{assume_Q} {\rm (cf. \cite[Assumption 2.1]{fxue_drs})}

{\rm (i) [Basic]} $\cQ:\cH\mapsto \cH$ is linear, bounded, self-adjoint and positive semi-definite, and $\domsf\cQ=\cH$;

{\rm (ii) [Degeneracy]} $\cQ$ is degenerate, i.e., $\kersf\cQ \backslash \{0\} \ne \varnothing$;

{\rm (iii) [Closedness]} $\cQ$ has a closed range, i.e., $\overline{\ransf \cQ} = \ransf \cQ$.
\end{assumption}

\begin{remark} \label{rmk_closed}
 A simple criterion for the closedness is: $\exists \alpha>0$, such that $\|x\|_\cQ^2 \ge \alpha \|x\|^2,\ \forall x\in \ransf\cQ$ \cite[Fact 2.26]{plc_book}.

\end{remark}


A few basic properties of $\cQ$ are listed here.
\begin{fact} \label{f_Q}
Under Assumption \ref{assume_Q}, the following hold.

{\rm (i) \cite[Fact 2.25]{plc_book}} An orthogonal decomposition of $\cH$ is: $\cH = \ransf \cQ +\kersf \cQ$.

{\rm (ii)  \cite[Proposition 3.30]{plc_book}} Any point $x\in\cH$ can be decomposed as $x=x_\rsf+x_\ksf$, where $x_\rsf :=\cP_{\ransf\cQ} x$ and $x_\ksf :=\cP_{\kersf \cQ}x$ are the $\ransf \cQ$- and $\kersf \cQ$-parts of $x$.

{\rm (iii)  \cite[Theorem 4]{bernau}} $\cQ$ admits an (actually unique) self-adjoint square-root, denoted as $\sqrt{\cQ}$, such that $\cQ = \sqrt{\cQ}\sqrt{\cQ}$. Moreover, $\ransf\cQ = \ransf \sqrt{\cQ}$.

{\rm (iv) \cite[Proposition 3.30-(ii), (iii)]{plc_book}} $\cP_{\ransf \cQ} = \cQ \cQ^\dagger$, $\cP_{\kersf \cQ} = \cI-  \cQ \cQ^\dagger$, where $\cQ^\dagger$ denotes the (Moore--Penrose) pseudo-inverse of $\cQ$. 
\end{fact}

\section{Basic properties under restricted monotonicity}
\label{sec_local}
To study the properties of $\cT$, let us now focus on the single step of \eqref{T} in this sequel, which is rewritten for simplicity as:
\be  \label{T_single}
y\in \cT x :=  (\cA+\cQ)^{-1} \cQ x. 
\ee
Here, $\cT$ is possibly a set-valued mapping, which maps an input $x\in\domsf\cT$ to a set $\cT x$. $y$ denotes a particular element in this set.

We first develop the solution set $\cT x$ of \eqref{T_single} in Lemma \ref{l_X}. It then gives rise to a restricted version of monotonicity (cf. Definition \ref{def_mon}), from which follow several nice properties such as firm nonexpansiveness and demiclosdedness in terms of $\cQ$-seminorm. These preliminary results  in Sect. \ref{sec_property} will be essential for Sect. \ref{sec_full}--\ref{sec_con} later.


\subsection{Characterization of the operator $\cT$}
An expression of solution set of  \eqref{T_single} is given below. This paves an easy and direct way for developing  later results.
\begin{lemma} \label{l_X}
Given \eqref{T_single} , the solution set $\cT x$ is given as
\[
\cT x = \bigcup_{y_\rsf\in Y_\rsf} \Big( y_\rsf + \big( (\cA^{-1} \cQ (x_\rsf - y_\rsf) - y_\rsf)   \mcap \kersf \cQ \big) \Big).
\]
Here, $Y_\rsf = (\cP_{\ransf \cQ} \circ \cT) x$ denotes the range component of $\cT x$, which is given as $Y_\rsf = (\sqrt{\cQ})^\dagger \big(\cI+  (\sqrt{\cQ} \cA^{-1} \sqrt{\cQ})^{-1} \big)^{-1}   (\sqrt{\cQ} x_\rsf) $. 
\end{lemma}
\begin{proof}
We develop:
\begin{eqnarray} \label{rr3}
&& y \in (\cA+\cQ)^{-1} \cQ x
\nonumber \\
& \Longleftrightarrow & 0\in \cA y +\cQ (y -x)
\nonumber \\
& \Longleftrightarrow & y \in \cA^{-1} \cQ (x - y)
\nonumber \\
& \Longleftrightarrow &
 \left\{ \begin{array}{l}
\sqrt{\cQ} y_\rsf \in \sqrt{\cQ} \cA^{-1} \sqrt{\cQ}  
(\sqrt{\cQ} x_\rsf - \sqrt{\cQ} y_\rsf)  \\ %
y_\ksf \in  \cA^{-1} \cQ (x_\rsf - y_\rsf)
- y_\rsf  \textrm{\ and\ } y_\ksf \in \kersf \cQ
\end{array} \right.
\nonumber \\
& \Longleftrightarrow &
 \left\{ \begin{array}{l}
(\sqrt{\cQ} \cA^{-1} \sqrt{\cQ})^{-1} (\sqrt{\cQ} y_\rsf )
\owns   \sqrt{\cQ} x_\rsf - \sqrt{\cQ} y_\rsf \\
y_\ksf \in  \big( \cA^{-1} \cQ (x_\rsf - y_\rsf)
- y_\rsf \big) \mcap \kersf \cQ
\end{array} \right.
\nonumber \\
& \Longleftrightarrow &
 \left\{ \begin{array}{l}
\sqrt{\cQ} y_\rsf \in \big(\cI+  (\sqrt{\cQ} \cA^{-1} \sqrt{\cQ})^{-1} \big)^{-1}   (\sqrt{\cQ} x_\rsf)  
:= J_{\sqrt{\cQ} \infimal \cA}    (\sqrt{\cQ} x_\rsf) \\
y_\ksf \in  \big( \cA^{-1} \cQ (x_\rsf - y_\rsf)
- y_\rsf \big) \mcap \kersf \cQ
\end{array} \right.
\end{eqnarray}
Combining both lines completes the proof.
\end{proof}

Recalling the first line of \eqref{rr3} and denoting $\cB := \sqrt{\cQ} \circ  \cA^{-1} \circ \sqrt{\cQ}$,  there is no need to additionally assume $\ransf ( \cI +  \cB^{-1} )^{-1} \mcap \ransf \cQ \ne\varnothing$ for existence of $\sqrt{\cQ} y_\rsf$. Actually, we can even obtain a stronger conclusion:  $\ransf ( \cI + \cB^{-1} )^{-1}  \subseteq \ransf \cQ$. To see this, let any $b\in  \ransf (\cI + \cB^{-1} )$, then we deduce that $a\in ( \cI + \cB^{-1} )^{-1}  b \Longleftrightarrow a+ \cB^{-1} a \owns b\Longleftrightarrow a\in \cB (b-a) \Longrightarrow a\in \ransf \cB \subseteq \ransf \cQ$. 

\subsection{Restricted monotonicity}
\label{sec_restricted}
Observing $Y_\rsf$ in Lemma \ref{l_X}, it motivates us to study the monotone property of $ (\sqrt{\cQ} \cA^{-1} \sqrt{\cQ})^{-1} $. First, this  operator can be simplified to several equivalent forms in terms of monotonicity.
\begin{lemma} \label{l_mon_eq}
The following five statements are equivalent.

{\rm (i)}  $ (\sqrt{\cQ} \cA^{-1} \sqrt{\cQ})^{-1} $ is monotone;

{\rm (ii)}  $\sqrt{\cQ} \cA^{-1} \sqrt{\cQ} $ is monotone;

{\rm (iii)} $\cP_{\ransf\cQ} \circ  \cA^{-1}|_{\ransf \cQ}$  is monotone;

{\rm (iv)} $\cA^{-1}|_{\ransf \cQ}$ is monotone; 

{\rm (v)} $\grasf \cA \mcap (\cH \times \ransf \cQ)$ is  monotone.
\end{lemma}
\begin{proof}
We show them following the order of (i) $\Longleftrightarrow$(ii)$\Longrightarrow$(iii)$\Longrightarrow$(iv)$\Longrightarrow$(v)$\Longrightarrow$(ii).

(i)$\Longleftrightarrow$(ii): \cite[Proposition 20.10]{plc_book}.

(ii)$\Longrightarrow$(iii): Given $(x,u) \in\grasf (\cP_{\ransf\cQ} \cA^{-1}|_{\ransf \cQ})$ and $(y,v) \in\grasf (\cP_{\ransf\cQ} \cA^{-1}|_{\ransf \cQ})$, then, $(x,y,u,v) \in \ransf \cQ =\ransf \sqrt{\cQ} = \ransf (\sqrt{\cQ})^\dagger$.
Letting $x=\sqrt{\cQ} x' $, $y=\sqrt{\cQ} y'$,  $u= (\sqrt{\cQ})^\dagger u'  $,  $v= (\sqrt{\cQ})^\dagger v' $ for some $u',v'\in\ransf\cQ$, we have $u = (\sqrt{\cQ})^\dagger u' \in \cP_{\ransf\cQ} \cA^{-1}|_{\ransf \cQ} x = \cP_{\ransf\cQ} \cA^{-1} \sqrt{\cQ} x'$ and  $v =(\sqrt{\cQ})^\dagger v' \in \cP_{\ransf\cQ} \cA^{-1}|_{\ransf \cQ} y = \cP_{\ransf\cQ} \allowbreak \cA^{-1} \sqrt{\cQ} y'$. This yields that $u' \in \sqrt{\cQ}  \cA^{-1} \sqrt{\cQ} x'$ and  $v' \in \sqrt{\cQ}  \cA^{-1} \sqrt{\cQ} y'$.  Consider
\begin{eqnarray}
\langle x-y| u-v  \rangle 
&= & \big\langle \sqrt{\cQ} x' - \sqrt{\cQ}  y' \big| 
(\sqrt{\cQ})^\dagger u' - (\sqrt{\cQ})^\dagger v'  \big\rangle 
\nonumber \\
&=&  \big\langle x' -  y' \big|  \cP_{\ransf \cQ} u' - 
 \cP_{\ransf \cQ} v'  \big\rangle \quad \textrm{by $\cP_{\ransf \cQ} = \sqrt{\cQ} (\sqrt{\cQ})^\dagger $}
 \nonumber \\
&=&  \big\langle  x' -  y' \big|  u' -  v'  \big\rangle 
\quad \textrm{by $(u',v') \in  \ransf \cQ$}
 \nonumber \\
& \ge & 0, \quad \textrm{by monotonicity of $ \sqrt{\cQ}  \cA^{-1} \sqrt{\cQ}$ assumed by (ii)}
\nonumber
\end{eqnarray}
which shows (iii).

(iii)$\Longrightarrow$(iv): Given $(x,u) \in\grasf (\cA^{-1}|_{\ransf \cQ})$ and $(y,v) \in\grasf (\cA^{-1}|_{\ransf \cQ})$, then, $(x,y) \in \ransf \cQ$. Consider
\begin{eqnarray}
\langle x-y| u-v  \rangle 
&= & \big\langle x -  y \big| 
\cP_{\ransf \cQ}  u -  \cP_{\ransf \cQ}  v  \big\rangle 
\quad \textrm{by $(x,y) \in  \ransf \cQ$}
\nonumber \\
&=&  \big\langle x -  y \big|u'-v'  \big\rangle \quad \textrm{letting $  u' = \cP_{\ransf \cQ} u$ and $v'= \cP_{\ransf \cQ} v $}
\nonumber \\
& \ge & 0, \quad \textrm{by $u'\in (\cP_{\ransf \cQ}\cA^{-1}|_{\ransf \cQ}) x$ and 
 $v' \in (\cP_{\ransf \cQ}\cA^{-1}|_{\ransf \cQ}) y$}
\nonumber
\end{eqnarray}
which shows (iv).

(iv)$\Longrightarrow$(v): Given $(x,u) \in\grasf \cA\mcap (\cH \times \ransf \cQ) $ and $(y,v) \in\grasf \cA\mcap (\cH \times \ransf \cQ) $, then, $u\in\cA x\mcap \ransf \cQ$ and $v\in\cA y \mcap \ransf \cQ$. This implies that $x\in \cA^{-1} u = (\cA^{-1}|_{\ransf \cQ} ) u$ and $y\in \cA^{-1} v = (\cA^{-1}|_{\ransf \cQ} ) v$. Thus, we have by (iv) that $\langle x-y| u-v  \rangle \ge 0$.

(v)$\Longrightarrow$(ii): Given $(x,u) \in\grasf (\sqrt{\cQ} \cA^{-1} \sqrt{\cQ}) $ and $(y,v) \in\grasf (\sqrt{\cQ} \cA^{-1} \sqrt{\cQ})$, 
then, $u\in\sqrt{\cQ} \cA^{-1} \sqrt{\cQ} x$ and $v\in \sqrt{\cQ} \cA^{-1} \sqrt{\cQ} y$.  This implies that $(\sqrt{\cQ})^\dagger u +s_\ksf\in \cA^{-1} \sqrt{\cQ} x$ and $(\sqrt{\cQ})^\dagger v +s'_\ksf\in \cA^{-1} \sqrt{\cQ} y$, for some $s_\ksf, s'_\ksf\in\kersf \cQ$. We further have $\sqrt{\cQ} x \in \cA ( (\sqrt{\cQ})^\dagger u +s_\ksf ) $ and $\sqrt{\cQ} y \in \cA ( (\sqrt{\cQ})^\dagger v +s'_\ksf ) $. Letting $x' = \sqrt{\cQ} x$, $y' = \sqrt{\cQ} y$,  $u' =  (\sqrt{\cQ})^\dagger u +s_\ksf $ and
$v' =  (\sqrt{\cQ})^\dagger v + s'_\ksf $, we have
$x'\in \cA u' \mcap \ransf \cQ$ and $y'\in \cA v' \mcap \ransf \cQ$.  Consider
\begin{eqnarray}
\langle x-y| u-v  \rangle 
&= & \big\langle  (\sqrt{\cQ})^\dagger x' +x_\ksf -  
(\sqrt{\cQ})^\dagger y'  -y_\ksf \big| 
\sqrt{\cQ}  u' -  \sqrt{\cQ}  v'  \big\rangle 
\nonumber \\
&=&  \big\langle \cP_{\ransf \cQ} x' -  
\cP_{\ransf \cQ} y' \big|u'-v'  \big\rangle
 \quad \textrm{by $x_\ksf, y_\ksf\in\kersf\cQ$}
\nonumber \\
&=&  \big\langle  x' - y' \big|u'-v'  \big\rangle
 \quad \textrm{by $x',y' \in\ransf\cQ$}
\nonumber \\
& \ge & 0, \quad \textrm{by monotonicity of $\grasf \cA \mcap (\cH \times \ransf \cQ)$}
\nonumber
\end{eqnarray} 
which shows (ii).  
\end{proof}

Lemma \ref{l_mon_eq} can be understood as an extension of \cite[Proposition 20.10]{plc_book} to the monotonicity restricted to a subspace. Observing Lemma \ref{l_mon_eq}-(v), the monotonicity of $\cA$ is associated with the degeneracy of $\cQ$, which gives rise to the following important definition. 
\begin{definition} [restricted monotonicity] \label{def_mon}
Assuming $\ransf\cA \mcap \ransf\cQ\ne \varnothing$, $\cA$ is monotone in $\cH \times \ransf\cQ$, or equivalently saying, $\grasf\cA \mcap (\cH \times \ransf \cQ)$ is monotone\footnote{Here, we define a set  (rather than an operator) to be monotone. This non-standard but brief terminology is for convenience of our further discussion. See \cite[Proposition 2.5]{bredies_ppa} for similar treatment.}, if
\[
\langle  x-y|u-v  \rangle \ge 0,\quad 
\forall (x,u), (y,v) \in \grasf \cA \mcap (\cH \times \ransf \cQ).
\]
\end{definition}

This notion was also mentioned in \cite[Proposition 2.5]{bredies_ppa}, and shown to be equivalent to the $\cQ$-monotonicity of $\cQ^{-1}\cA$ \cite[Definition 2.4]{bredies_ppa}.  Thus, we put the assumption of $\cA$ as
\begin{assumption} \label{assume_local}
{\rm (i) [Basic viability]} $\ransf \cA \mcap \ransf \cQ \ne \varnothing$;

{\rm (ii) [Restricted monotonicity]} $\grasf \cA \mcap (\cH \times \ransf\cQ)$ is monotone.
\end{assumption}
\begin{remark}
The condition of $\ransf \cA  \mcap \ransf\cQ  \ne \varnothing$ has been mentioned in Definition \ref{def_mon}, and this is a basic viability of all the operators listed in Lemma \ref{l_mon_eq}.
\end{remark}


\subsection{Basic properties} \label{sec_property}
Now, we show several basic properties on $\cT$ based on Assumptions \ref{assume_Q} and \ref{assume_local}.  The first  is the single-valuedness of $\cP_{\ransf\cQ} \circ \cT$. 
 
\begin{lemma} \label{l_single}
Under Assumptions \ref{assume_Q} and   \ref{assume_local}, $\cP_{\ransf\cQ} \circ \cT$ is single-valued.
\end{lemma}
\begin{proof}
Let $x \in \domsf \cT$ and suppose that $y_1 \in (\cA+\cQ)^{-1} \cQ x \subseteq \ransf\cT$ and  $y_2 \in (\cA+\cQ)^{-1} \cQ x \subseteq \ransf\cT$. Then, we have 
\be  \label{t4}
\cQ(x-y_1) \in \cA y_1, \quad  \cQ(x-y_2) \in \cA y_2.
\ee
 Thus, by Assumptions \ref{assume_Q} and   \ref{assume_local}, we derive that
\begin{eqnarray}
0 & \ge & -\|y_1-y_2\|_\cQ^2 
= \big\langle y_1-y_2 | \cQ (y_2-y_1) \big\rangle 
\quad  
\nonumber \\
&=& \big\langle y_1-y_2 | \cQ (x-y_1) - \cQ (x-y_2) \big\rangle \ge 0.
\nonumber
\end{eqnarray}
By  \eqref{t4} and monotonicity of $\grasf\cA\mcap (\ransf \cT \times \ransf \cQ)$, it  yields $\| y_1 - y_2 \|_\cQ = 0$, which implies $ y_1 -  y_2 \in\kersf \cQ$ and thus, $\cP_{\ransf\cQ} y_1=\cP_{\ransf\cQ} y_2$.
\end{proof}

Alternatively, Lemma \ref{l_single} can also be obtained by Lemma \ref{l_X}, by showing the uniqueness of $\sqrt{\cQ} y_\rsf$ in  \eqref{rr3}, based on the monotonicity of $\sqrt{\cQ} \infimal \cA$ (cf. \cite[Proposition 20.10]{plc_book}).

\cite[Proposition 4.2]{hhb_resolvent} showed that $\cT$ is $\cQ$-firmly nonexpansive under maximally monotone $\cA: \cH\mapsto 2^\cH$ and non-degenerate $\cQ$. This was extended to degenerate case in \cite[Lemma 3.2]{bredies_2017} and \cite[Sect. 2]{fxue_rima}.  \cite[Lemma 2.6]{bredies_ppa} further relaxed the condition to the  restricted monotonicity of $\grasf\cA \mcap (\cH \times \ransf\cQ)$. Here we show that this key property is also preserved under a milder restricted monotonicity of $\grasf\cA \mcap (\ransf\cT \times \ransf\cQ)$.
\begin{fact} \label{f_fne}
Under Assumptions \ref{assume_Q} and \ref{assume_local}, $\cT$ given as \eqref{T} is $\cQ$-firmly nonexpansive, i.e.,
\[
\big\|\cT x-\cT y\big\|_\cQ^2 \le \big\langle \cT x-\cT y \big| \cQ(x-y) \big\rangle, \quad \forall x,y \in\domsf \cT,
\]
or equivalently, 
\be  \label{fne}
\big\|\cT x-\cT y\big\|_\cQ^2+
\big\| (\cI-\cT) x- (\cI-\cT) y\big\|_\cQ^2 \le 
\big\|x-y\big\|_\cQ^2, \quad \forall x,y \in\domsf \cT.
\ee
\end{fact}
\begin{remark}
There is still a slight difference between Fact \ref{f_fne} and \cite[Lemma 2.6]{bredies_ppa}, \cite[Lemma 3.2]{bredies_2017}. The latter two results required $\cT$ to be well-defined, which, however, is not assumed in our result. Nevertheless,   we clarify that all the expressions in above inequalities are well-defined,  even if $\cT$ is multi-valued. Indeed, $\|\cT x\|_\cQ=\|\cP_{\ransf\cQ}\cT x\|_\cQ$ and $\langle \cT x|\cQ y\rangle =\langle \cP_{\ransf\cQ}\cT x|\cQ y\rangle$, $\forall (x,y)\in \domsf\cT \times \cH$. They are well-defined, since  $\cP_{\ransf\cQ} \circ \cT$ is single-valued (by Lemma \ref{l_single}). To put it more explicitly, combining with $\cT = \cT\circ \cP_{\ransf \cQ}$, \eqref{fne} can be equivalently rewritten as $\forall x_\rsf, y_\rsf \in \cP_{\ransf\cQ} (\domsf \cT)$:
\begin{eqnarray}   \label{fne_r}
\big\|x_\rsf - y_\rsf \big\|_\cQ^2
& \ge & \big\| (\cP_{\ransf\cQ} \circ\cT)( x_\rsf)-
(\cP_{\ransf\cQ}\circ \cT) (y_\rsf) \big\|_\cQ^2
\nonumber \\
&+ &
\big\| (\cI-  \cP_{\ransf\cQ} \circ\cT)( x_\rsf)- (\cI- \cP_{\ransf\cQ} \circ\cT) (y_\rsf) \big\|_\cQ^2.
\end{eqnarray}
\end{remark}

The concept of demiclosedness \cite[Definition 4.26]{plc_book} and Browder's demiclosedness principle \cite[Theorem 4.27]{plc_book} are crucial for convergence analysis, see \cite[Lemma 3.4, Theorem 3.1]{bredies_2017} for example. 
 Here it is a natural way (or even trivial) to extend the results to any $\cQ$-firmly nonexpansive operator $\cS$.
\begin{definition} [$\cQ$-demiclosedness] \label{def_demi}
Under Assumption \ref{assume_Q}, let $D$ be a nonempty weakly sequentially closed subset of
$\ransf\cQ$, let $\cS : D \mapsto \ransf\cQ$, and let $u \in \ransf\cQ$. Then $\cS$ is $\cQ$-demiclosed at $u$ if, for every
sequence $\big\{x^k \big\}_{k\in\N}$ in $D$ and every $x \in D$ such that $x^k\weak x$ and $\cS x^k \rightarrow  u$, we have $\cS x = u$. In addition, $\cS$ is $\cQ$-demiclosed if it is demiclosed at every point
in $D$.
\end{definition}

\begin{lemma} [$\cQ$-demiclosedness principle] \label{l_demi}
Under Assumption \ref{assume_Q}, let $D$ be a nonempty weakly sequentially closed subset of $\ransf\cQ$ and let $\cS: D \mapsto \ransf\cQ$ be a $\cQ$-nonexpansive operator, i.e.,
\[
\big\|\cS x_1-\cS x_2 \big\|_\cQ \le 
\big\|x_1-x_2 \big\|_\cQ, \quad
\forall x_1,x_2\in D \subseteq\ransf\cQ \cap \domsf\cS.
\]
Then $\cI-\cS$ is $\cQ$-demiclosed.
\end{lemma}

The proof is very similar to  \cite[Theorem 4.27]{plc_book}, by replacing ordinary norm $\|\cdot\|$ by semi-norm $\|\cdot\|_\cQ$ and treating $\ransf\cQ$ as the ambient space.  The $\cQ$-demiclosedness essentially confines the demiclosed property within the subspace $\ransf\cQ$, which indicates that  the weak and strong convergences take place in $\ransf\cQ$ only.  This restriction makes sense: Lemma \ref{l_demi_T} will  claim $\cI-\cP_{\ransf\cQ}\circ\cT|_{\ransf\cQ}$ is $\cQ$-demiclosed, but $\cI-\cT$ is not necessarily so.

\begin{lemma} [$\cQ$-demiclosedness of $\cI-  \cP_{\ransf\cQ} \circ \cT|_{\ransf\cQ}$] \label{l_demi_T}
 Given $\cT$ as \eqref{T} satisfying Assumptions \ref{assume_Q} and   \ref{assume_local}, let $D$ be a nonempty closed subset of $\ransf\cQ \cap \domsf\cT$, then $\cI-  \cP_{\ransf\cQ} \circ \cT|_{\ransf\cQ}: D  \mapsto \ransf\cQ$ is $\cQ$-demiclosed.
\end{lemma}
\begin{proof}
In view of Lemma \ref{l_demi}, it suffices to show that  $ \cP_{\ransf\cQ} \circ \cT|_{\ransf\cQ}: D  \mapsto \ransf\cQ $ is $\cQ$-nonexpansive. Indeed, this immediately follows from  Fact \ref{f_fne}.
\end{proof}

Proposition \ref{p_equality} develops several useful results regarding \eqref{T_single} under the restricted monotonicity. Note that $\cT$ does not necessarily have full domain for the moment, and $\sqrt{\cQ} \infimal \cA$ may not be maximally monotone, and thus  its resolvent is not guaranteed to be well-defined. In this sense, Proposition \ref{p_equality}-(vi)  generalizes the classic Moreau's decomposition identity \cite[Proposition 23.20]{plc_book}.  

\begin{proposition} \label{p_equality}
Given \eqref{T_single} under Assumptions \ref{assume_Q} and  \ref{assume_local}, the following  hold.

{\rm (i)} $J_{\sqrt{\cQ} \infimal \cA} = \big(\cI+  (\sqrt{\cQ} \cA^{-1} \sqrt{\cQ})^{-1} \big)^{-1}$ is single-valued.

{\rm (ii)} $Y_\rsf = \{y_\rsf\} =  \big\{ (\sqrt{\cQ})^\dagger \big(\cI+  (\sqrt{\cQ} \cA^{-1} \sqrt{\cQ})^{-1} \big)^{-1}   (\sqrt{\cQ} x_\rsf) \big\}$.

{\rm (iii)} Given $x \in \domsf\cT$, the solution set $\cT x$ is given as 
\[
\cT x  = y_\rsf + \big( \cA^{-1} \cQ (x_\rsf - y_\rsf)
- y_\rsf\big) \mcap \kersf \cQ,
\]
where $y_\rsf $ is given in  (ii). 

{\rm (iv)} $\cT x\ne\varnothing$, if and only if $y_\rsf$ exists. 

{\rm (v)} $\domsf\cT = \domsf\cT +\kersf \cQ= \cP_{\ransf\cQ} (\domsf\cT) +\kersf \cQ$.

{\rm (vi)} $\cI = \big(\cI+  (\sqrt{\cQ} \cA^{-1} \sqrt{\cQ})^{-1} \big)^{-1} + \big(\cI+  \sqrt{\cQ} \cA^{-1} \sqrt{\cQ} \big)^{-1}  $.

{\rm (vii)} $\domsf \cT = (\sqrt{\cQ})^\dagger \ransf \big(\cI+   (\sqrt{\cQ} \cA^{-1} \sqrt{\cQ} )^{-1} \big)  +\kersf\cQ $. 

{\rm (viii)}  $\ransf \big(\cI+  (\sqrt{\cQ} \cA^{-1} \sqrt{\cQ})^{-1} \big)  = \ransf \big(\cI+  (\sqrt{\cQ} \cA^{-1} \sqrt{\cQ})^{-1} \big) + \kersf\cQ$.

{\rm (ix)} If $\ransf \big(\cI+  (\sqrt{\cQ} \cA^{-1} \sqrt{\cQ})^{-1} \big)  =\cH$, then $\ransf \big(\cI+  (\sqrt{\cQ} \cA^{-1} \sqrt{\cQ}|_{\ransf\cQ})^{-1} \big) 
\allowbreak  = \ransf\cQ$. 
\end{proposition}
\begin{proof}
(i)--(iii): in view of Lemmata \ref{l_X} and \ref{l_single}.

(iv) This statement is equivalent to $\big( \cA^{-1} \cQ (x_\rsf - y_\rsf) - y_\rsf \big) \mcap \kersf \cQ \ne\varnothing$, if $y_\rsf$ exists. Indeed, Combining with the proof of Lemma \ref{l_X}, $y_\rsf$ must satisfy $y_\rsf \in  \cP_{\ransf \cQ}  \cA^{-1} \cQ (x_\rsf-y_\rsf )$.
This implies that $\exists y_\ksf  \in \kersf \cQ$, such that $y_\rsf  +  y_\ksf  \in \cA^{-1} \cQ (x_\rsf - y_\rsf)$, which completes the proof.

(v) Noting that $\cT = \cT \circ \cP_{\ransf\cQ}$, we develop $\domsf\cT = \domsf (\cT \circ \cP_{\ransf\cQ} ) = 
\big\{ x\in \cH:  \cP_{\ransf\cQ} x \in \domsf\cT \big\} = \big\{ x+s\in \cH: \cP_{\ransf\cQ}  (x+s)\in  \domsf\cT, \forall s\in\kersf \cQ \big\}  = \domsf\cT +\kersf\cQ $.  The second equality is due to $\cP_{\ransf\cQ}  x = \cP_{\ransf\cQ} ( x+s) = \cP_{\ransf\cQ}  (x_\rsf+s)$, $\forall s\in \kersf \cQ$.

(vi) Denoting $\cB = \sqrt{\cQ} \circ \cA^{-1} \circ \sqrt{\cQ}$ (which is monotone, but not necessarily maximally monotone for the moment), we develop,  $\forall a\in \domsf  (\cI +\cB^{-1})^{-1}$, that $b = (\cI +\cB^{-1})^{-1} a \Longleftrightarrow a\in b + \cB^{-1} b \Longleftrightarrow b\in \cB(a-b)\Longleftrightarrow a-(a-b) \in \cB(a-b) \Longleftrightarrow a-b = (\cI+\cB)^{-1} a \Longleftrightarrow b = a - (\cI+\cB)^{-1} a = 
(\cI  - (\cI+\cB)^{-1} ) a$.

(vii) Combining (iv), (v) and \eqref{rr3}, we deduce that
\begin{eqnarray}
\domsf\cT &=& \{x\in\cH: \cT x \ne \varnothing\} 
\nonumber \\
&=& \{x_\rsf\in\ransf\cQ:  \cT x_\rsf \ne \varnothing\} +\kersf\cQ
\nonumber \\
&=& \big\{ x_\rsf\in\ransf\cQ:  
\sqrt{\cQ} x_\rsf \in \domsf \big(\cI+  (\sqrt{\cQ} \cA^{-1} \sqrt{\cQ})^{-1} \big)^{-1} \big\} +\kersf\cQ
\nonumber \\
&=&
 \big\{ x_\rsf\in\ransf\cQ:  
\sqrt{\cQ} x_\rsf \in \ransf \big(\cI+  (\sqrt{\cQ} \cA^{-1} \sqrt{\cQ})^{-1} \big) \big\} +\kersf\cQ 
\nonumber \\
&=& (\sqrt{\cQ})^\dagger \ransf \big(\cI+  (\sqrt{\cQ} \cA^{-1} \sqrt{\cQ})^{-1} \big)  +\kersf\cQ.
\nonumber 
\end{eqnarray}

(viii) It suffices to show that for any $t\in \ransf \big(\cI+  (\sqrt{\cQ} \cA^{-1} \sqrt{\cQ})^{-1} \big)$, we have $t+ \kersf \cQ \subseteq \ransf \big(\cI+  (\sqrt{\cQ} \cA^{-1} \sqrt{\cQ})^{-1} \big)$. Since $t\in  \ransf \big(\cI+  (\sqrt{\cQ} \cA^{-1} \sqrt{\cQ})^{-1} \big)$, $\exists q\in \ransf \cQ$, such that $t \in q + (\sqrt{\cQ} \cA^{-1} \sqrt{\cQ})^{-1}  q$. It yields that $(\sqrt{\cQ} \cA^{-1} \sqrt{\cQ}) (t-q) \owns q$, and thus,  $(\sqrt{\cQ} \cA^{-1} \sqrt{\cQ}) (t+s-q) \owns q$, $\forall s\in \kersf Q$, since $\sqrt{\cQ} s = 0$. This implies that $t+s\in  q + (\sqrt{\cQ} \cA^{-1} \sqrt{\cQ})^{-1}  q \subseteq \ransf \big(\cI+  (\sqrt{\cQ} \cA^{-1} \sqrt{\cQ})^{-1} \big)$, $\forall s\in \kersf \cQ$.  

(ix)  Denote $\cB = \sqrt{\cQ} \circ \cA^{-1} \circ  \sqrt{\cQ} $ for brevity. If  $\ransf \big(\cI+  \cB^{-1} \big)  =\cH (\supseteq \ransf\cQ)$, this  implies that $\forall a \in \ransf\cQ$,   $\exists b \in \domsf \cB^{-1} \subseteq \ransf \cQ$, such that $(\cI+\cB^{-1})b = b+\cB^{-1} b \owns a$. This further indicates that $\cB^{-1} b\subseteq\ransf\cQ$ for such $b$.  Denoting $c\in \cB^{-1}b$ for such $b$, then $c\in\ransf\cQ$, and $\cB c\owns b$. This can be rewritten as $(\cB|_{\ransf\cQ}) (c)\owns b$, and further $c\in (\cB|_{\ransf\cQ})^{-1} b$. Thus, we obtain that $\cB^{-1} b =  (\cB|_{\ransf\cQ})^{-1} b$ for such $b$. This shows that $\forall a \in \ransf\cQ$, $\exists b\in \domsf  (\cB|_{\ransf\cQ})^{-1}\subseteq \ransf\cQ$,  such that $(\cI+ (\cB|_{\ransf\cQ})^{-1}) b  \owns a$,  i.e.,  $\ransf \big(\cI+   (\cB|_{\ransf\cQ})^{-1} \big)   \supseteq \ransf\cQ$. On the other hand, it is clear that  $\ransf \big(\cI+   (\cB|_{\ransf\cQ})^{-1} \big)   \subseteq \ransf\cQ$, since $\domsf (\cB|_{\ransf\cQ})^{-1}  \subseteq\ransf\cQ$ and  $\ransf (\cB|_{\ransf\cQ})^{-1}  \subseteq\ransf\cQ$.  
\end{proof}

\section{Full domain of degenerate preconditioned resolvent}  \label{sec_full}
We first present a sufficient and necessary condition on $\cA$ and $\cQ$ for the full domain of $\cT$, which is further strengthened to a stronger condition with global maximal monotonicity of $\cA$.

\subsection{Main result}
Proposition \ref{p_equality}-(vi) inspires us to study maximality of the monotonicity of $ (\sqrt{\cQ} \cA^{-1} \sqrt{\cQ})|_{\ransf\cQ} $, which can be equivalently simplified as that of $\cP_{\ransf\cQ}\cA^{-1}|_{\ransf\cQ}:\ransf\cQ \mapsto 2^{\ransf\cQ}$, as shown below. 
\begin{lemma} \label{l_max_mon}
The maximal monotonicity of $\cP_{\ransf\cQ} \circ \cA^{-1}|_{\ransf\cQ}:\ransf\cQ \mapsto 2^{\ransf\cQ}$ is equivalent to that of $ (\sqrt{\cQ} \cA^{-1} \sqrt{\cQ})|_{\ransf\cQ}:\ransf\cQ \mapsto 2^{\ransf\cQ}$.
\end{lemma}
\begin{proof}
First, if $\cP_{\ransf\cQ} \circ \cA^{-1}|_{\ransf\cQ}$ is maximally monotone, then for any given $(u,x)\in \ransf\cQ \times \ransf\cQ$, we have: 
\be \label{r5}
(u,x)\in \grasf (\cP_{\ransf\cQ} \cA^{-1}|_{\ransf\cQ}) \Longleftrightarrow 
\forall (v,y)\in \grasf (\cP_{\ransf\cQ} \cA^{-1}|_{\ransf\cQ}), \ 
\langle u-v|x-y \rangle \ge 0.
\ee
Let $u=\sqrt{\cQ}u'$ and $v=\sqrt{\cQ}v'$ for $u',v'\in\ransf\cQ$. Then, $x\in \cP_{\ransf\cQ}\cA^{-1} \sqrt{\cQ} u'$ and $y\in \cP_{\ransf\cQ}\cA^{-1} \sqrt{\cQ} v'$. Further, let $x'=\sqrt{\cQ} x$ and $y'=\sqrt{\cQ}y$, then \eqref{r5} becomes
\[
(u',x')\in \grasf \Big( (\sqrt{\cQ} \cA^{-1}\sqrt{\cQ} ) \big|_{\ransf\cQ} \Big) \Longleftrightarrow 
\forall (v',y')\in \grasf \Big( (\sqrt{\cQ} \cA^{-1}\sqrt{\cQ}) \big|_{\ransf\cQ} \Big), \ 
\langle u'-v' | x'-y' \rangle \ge 0.
\]
This shows that  $(\sqrt{\cQ}\cA^{-1}\sqrt{\cQ} ) \big|_{\ransf\cQ}$ is maximally monotone.  The converse statement can be shown in a similar manner.
\end{proof}

The maximality of $\cP_{\ransf\cQ} \circ \cA^{-1}|_{\ransf\cQ}$ or  $ (\sqrt{\cQ} \cA^{-1} \sqrt{\cQ})|_{\ransf\cQ} $  is essentially a restricted version of maximal monotonicity from $\cH$ to $2^\cH$, by confining both domain and range of the operator to $\ransf\cQ$. Indeed, this {\it restricted maximal monotonicity} is closely related to Definition \ref{def_mon}  due to Lemma \ref{l_mon_eq}. More precisely, Lemma \ref{l_minty} formally develops a restricted version of classic Minty's theorem \cite[Theorem 21.1]{plc_book} for the restricted maximality of  $\cP_{\ransf\cQ}  \circ \cA^{-1}|_{\ransf \cQ}$.
\begin{lemma}  [Restricted Minty's theorem] \label{l_minty}
Let  $\cP_{\ransf\cQ}  \circ \cA^{-1}|_{\ransf \cQ}:\ransf\cQ \mapsto 2^{\ransf\cQ}$ be monotone. Then  $\cP_{\ransf\cQ}  \circ \cA^{-1}|_{\ransf \cQ}:\ransf\cQ \mapsto 2^{\ransf\cQ}$  is maximally monotone, if and only if $\ransf (\cI+ \cP_{\ransf\cQ} \circ  \cA^{-1}|_{\ransf \cQ}) =\ransf\cQ$.
\end{lemma}

The proof is exactly the same as \cite[Theorem 21.1]{plc_book}, if the subspace $\ransf\cQ$ is regarded as the ambient space. Actually, our  specific $\cP_{\ransf\cQ} \circ \cA^{-1}|_{\ransf\cQ}$ in Lemma \ref{l_minty} can be replaced by any monotone operator from $\ransf\cQ$ to $2^{\ransf\cQ}$. Lemma \ref{l_minty} is then rephrased as the following assumption.
\begin{assumption} \label{assume_max}
 $\cP_{\ransf\cQ} \circ  \cA^{-1}|_{\ransf\cQ}: \ransf\cQ \mapsto 2^{\ransf\cQ}$ is maximally monotone.
\end{assumption}

Notice that Assumption \ref{assume_local}-(ii) has been fully implied by Assumption \ref{assume_max}. 
Based on Assumption \ref{assume_max}, We are ready to present the main result for the full domain of $\cT$.
\begin{theorem} \label{t_full}
Under Assumptions \ref{assume_Q} and \ref{assume_local}, $\cT$ given as \eqref{T} has full domain (i.e., $\domsf \cT =\cH$), if and only if Assumption \ref{assume_max} is satisfied.
\end{theorem}
\begin{proof}
We develop
\begin{eqnarray}
&& \domsf \cT = \cH  = \overline{\ransf}\cQ +\kersf\cQ = \ransf\cQ +\kersf\cQ 
\textrm{\ (by closedness of $\ransf\cQ$)}
\nonumber \\
& \Longleftrightarrow & (\sqrt{\cQ})^\dagger \ransf \big(\cI+   \sqrt{\cQ} \cA^{-1} \sqrt{\cQ}  \big)  +\kersf\cQ  = \ransf\cQ +\kersf\cQ 
\quad \textrm{by Proposition \ref{p_equality}-(viii)} 
\nonumber \\
& \Longleftrightarrow & (\sqrt{\cQ})^\dagger \ransf \big(\cI+   \sqrt{\cQ} \cA^{-1} \sqrt{\cQ}  \big)  = \ransf\cQ  
\quad \textrm{by $\ransf (\sqrt{\cQ})^\dagger = \ransf\cQ$} 
\nonumber \\
& \Longleftrightarrow &    \ransf \big(\cI+  ( \sqrt{\cQ} \cA^{-1} \sqrt{\cQ} )^{-1} \big)  = \ransf\cQ +\kersf \cQ = \cH  
\quad \textrm{by Proposition \ref{p_equality}-(viii)} 
\nonumber \\
& \Longleftrightarrow & \ransf \big(\cI+   \sqrt{\cQ} \cA^{-1} \sqrt{\cQ} |_{\ransf \cQ} \big)  = \ransf\cQ  
\quad \textrm{by Proposition \ref{p_equality}-(ix)} 
\nonumber \\
& \Longleftrightarrow & 
\cP_{\ransf\cQ}\cA^{-1}|_{\ransf\cQ}:\ransf\cQ \mapsto 2^{\ransf\cQ} \textrm{\ is maximally monotone}
\quad \textrm{by Lemmas \ref{l_max_mon} and \ref{l_minty}} 
\nonumber 
\end{eqnarray}
\end{proof}

Also note that the non-degenerate case can be easily recovered from Theorem \ref{t_full}.  This is also an extension of Minty's Theorem  \cite[Theorem 21.1]{plc_book}.
\begin{corollary}
Given a monotone operator $\cA:\cH\mapsto 2^\cH$ and non-degenerate $\cQ$, then $\cT$ has full domain, if and only if $\cA:\cH\mapsto 2^\cH$ is maximally monotone. 
\end{corollary}
\begin{proof}
In this case, $\ransf\cQ = \cH$,  $\cP_{\ransf\cQ}=\cI$ and $\cP_{\ransf\cQ} \circ \cA^{-1}|_{\ransf\cQ} = \cA^{-1}$. Then the proof is completed by Theorem \ref{t_full} and \cite[Proposition 20.22]{plc_book}.
\end{proof}

\subsection{Sufficient condition}
The next question regarding Theorem \ref{t_full} naturally arises: under what conditions on $\cA$ and $\cQ$,  $\sqrt{\cQ}\cA^{-1} \sqrt{\cQ}|_{\ransf\cQ}:\ransf\cQ \mapsto 2^{\ransf\cQ}$ is maximally monotone?  To answer this question, we propose the following assumption and derive a sufficient condition in the subsequent corollary.
\begin{assumption} \label{assume_A}
{\rm (i)} $\cA:\cH\mapsto 2^\cH$ is maximally monotone;
 
{\rm (ii)}  $0\in \srisf (\ransf\cQ - \ransf\cA) $, where the strong relative interior $(\srisf)$ is defined in \cite[Definition 6.9]{plc_book}.
\end{assumption}

\begin{corollary} \label{c_full}
Under Assumption \ref{assume_Q}, 
$\sqrt{\cQ}\cA^{-1}\sqrt{\cQ} \big|_{\ransf\cQ}$  is  maximally monotone, if Assumption \ref{assume_A} is fulfilled.
\end{corollary}
\begin{proof}
In view of \cite[Corollary 25.6 or Proposition 25.41-(iv)]{plc_book}.
\end{proof}

Corollary \ref{c_full} claims  the implication of  Assumption \ref{assume_A}$\Longrightarrow$Assumption \ref{assume_max}, but the converse is not true.  We stress the significance of Assumption \ref{assume_A}-(ii). Without this,  there is no straightforward implication between Assumption \ref{assume_max} and Assumption \ref{assume_A}-(i). Indeed, the maximal monotonicity of $\cA$ does not imply that of  $\cP_{\ransf\cQ} \circ \cA^{-1}|_{\ransf\cQ}$, and vice versa. This will be shown in later examples.

In addition, without Assumption \ref{assume_A}-(ii), the next proposition claims that $\cT$ `{\it almost}' has full domain. 
\begin{proposition} \label{p_close}
Under Assumptions \ref{assume_Q}, \ref{assume_local} and \ref{assume_A}-(i), if $0\in \ransf \cA$, then $\overline{\ransf}(\cA+\cQ)  \supseteq \ransf\cQ$.
\end{proposition}
\begin{proof}
First, $\cQ$ is maximally monotone by \cite[Corollary 20.28]{plc_book}, and  $3^*$ monotone by \cite[Example 25.17]{plc_book}. Then, $\cA+\cQ$ is maximally monotone by \cite[Corollary 25.5]{plc_book}. Then, noting $\domsf\cA \subseteq \domsf\cQ=\cH$ as Assumption \ref{assume_Q},   \cite[Theorem 25.24]{plc_book} yields
\[
\overline{\ransf}(\cA+\cQ) = \overline{\ransf \cA +\ransf \cQ} \supseteq  \overline{0 +\ransf \cQ} =  \overline{\ransf} \cQ =\ransf\cQ,
\]
which is due to $0\in \ransf\cA$ and the closedness of $\ransf\cQ$.
\end{proof}

The above result  holds, without Assumption \ref{assume_A}-(ii). 
However, it does not necessarily hold that $\ransf (\cA+\cQ) \supseteq \ransf\cQ$.  Unfortunately, it is generally difficult to investigate the relation between  $\overline{\ransf}(\cA+\cQ)$ and $\ransf(\cA+\cQ)$. We believe there are some underlying connections between Assumption \ref{assume_A}-(ii) and this relation. This is left to future work. 

\subsection{Maximality of restricted monotonicity} \label{sec_max}
The restricted monotonicity is closely related to the {\it local monotonicity} \cite[Definition 25.7]{plc_book}, which requires $\grasf\cA \mcap (\cH \times C)$ to be  monotone with the set $C$ being open. This restricted version  extends this concept by confining the monotonicity within a subspace instead of an open neighbourhood.
The essential difference between the restricted and local versions of monotonicity lies in the fact that $\ransf\cQ$ is not open, and thus, $\intsf\ \ransf\cQ =\varnothing$. 

The restricted monotonicity is also somewhat equivalent to the so-called {\it $\cQ$-monotonicity} (adapted to our notation, the same below) proposed in \cite[Definition 2.4, Proposition 2.5]{bredies_ppa} and \cite[Definition 2.5]{naldi_thesis}. This notion is based on a preimage of degenerate preconditioner $\cQ$. Parallel to \cite[Proposition 2.6]{naldi_thesis}, Lemma \ref{l_mon_eq} equivalently transform the monotonicity of  $\grasf \cA \mcap (\cH \times \ransf\cQ)$ to that of $\cA^{-1}|_{\ransf\cQ}$. However, the next proposition shows that it  generally fails to further enforce maximality of the monotonicity of $\cA^{-1}|_{\ransf\cQ}$.
\begin{proposition} \label{p_fail}
Even if $\cP_{\ransf\cQ} \circ \cA^{-1}|_{\ransf\cQ}$ is maximally monotone, the monotonicity of  $\cA^{-1}|_{\ransf\cQ}$ cannot be maximal, unless $\cP_{\kersf\cQ} \cA^{-1} x =\kersf\cQ$, $\forall x\in \ransf\cQ \mcap \ransf\cA$.
\end{proposition}
\begin{proof}
Fix $(x,u)\in \ransf\cQ\times \cH$, such that $\langle x-y|u-v\rangle \ge 0$, $\forall (y,v)\in \grasf(\cA^{-1}|_{\ransf\cQ})$. We claim that it generally fails to obtain that $u \in (\cA^{-1}|_{\ransf\cQ}) (x)$. Indeed, observing $x,y\in\ransf\cQ$, the inner product becomes $\langle x-y|u_\rsf - v_\rsf \rangle \ge 0$. By maximal monotonicity of 
$\cP_{\ransf\cQ} \circ \cA^{-1}|_{\ransf\cQ}$ (as assumed), we have $u_\rsf \in \cP_{\ransf\cQ} \cA^{-1} x$. Then, $u$ should be expressed as $u=u_\rsf +s$  with any  $ s\in\kersf\cQ$, i.e., $u\in \cP_{\ransf\cQ}
\cA^{-1} x+\kersf\cQ$. Thus, it does not generally yield that $u\in \cA^{-1} x$, unless $\cP_{\ransf\cQ} \cA^{-1} x+\kersf\cQ \subseteq\cA^{-1} x$.
\end{proof}

The above reasoning shows that the failure is due to the `{\it invalid}' inner product, which performs in a proper subspace only. This inspires us to define the maximality restricted to $\ransf\cQ$ in the range of $\cA^{-1}|_{\ransf\cQ}$ as well, in order to remove the redundancy of the inner product. More precisely, considering $u\in (\cA^{-1}|_{\ransf\cQ} ) x$, we  wish to restrict $u$ in $\ransf\cQ$, and discard the kernel component $u_\ksf$.  This gives us this new operator---$ \cP_{\ransf\cQ} \circ \cA^{-1}|_{\ransf\cQ}$, of which the maximal monotonicity can then be properly defined, and finally connected to the full domain of $\cT$ due to Lemma \ref{l_max_mon} and Theorem \ref{t_full}. 

By comparison, the works \cite{bredies_ppa,naldi_thesis} developed a different route. They explore the so-called $\cQ$-monotonicity of $\cQ^{-1} \circ \cA$, which relies to a large extent on $\cQ^{-1}$. To address this maximality issue, \cite[Definition 2.15]{naldi_thesis} proposed the $\cQ$-quotient maximal monotonicity to overcome the difficulty of using $\cQ$-maximal monotonicity \cite[Proposition 2.9]{naldi_thesis} . 
There are a few advantages of using $\cA^{-1}$ in this work over using $\cQ^{-1}$ in \cite{bredies_ppa,naldi_thesis}:

(1) The set-valued operators encountered in our expositions include $\cA$ and $\cA^{-1}$ only. We avoid to use $\cQ^{-1}$ as a new multivalued operator, and do not need to define a complex composition of $\cQ^{-1} \circ \cA$ (cf. \cite[Sect. 2]{bredies_ppa} and \cite[Sect. 2.1.1]{naldi_thesis}). We avoid to define the new $\cQ$-monotonicity and further $\cQ$-quotient maximal monotonicity \cite[Definition 2.15]{naldi_thesis}.

(2) The operator $\cA^{-1}$ is naturally introduced in our approach, when developing an expression of $\cT x$ in Lemma \ref{l_X}. This paves an easy and straightforward way for investigating basic properties of $\cT$.

(3) It is natural to extend the Minty's theorem to the degenerate case based on the maximal monotonicity of  $ \cP_{\ransf\cQ} \circ \cA^{-1}|_{\ransf\cQ}$, without introducing new definitions as in \cite[Theorem 2.16-(ii)]{naldi_thesis}. All the concepts we used are standard in convex analysis.

\subsection{Examples} \label{sec_eg}
We  list several examples with degenerate $\cQ$ to show

(1) $\cA+\cQ$ is often non-surjective (see Example \ref{eg_1});

(2) The maximal monotonicity of $\cA$ does not necessarily imply that of  $ \sqrt{\cQ}\cA^{-1}\sqrt{\cQ}  \big|_{\ransf\cQ} $, and vice versa (see Examples \ref{eg_2}, \ref{eg_4}, \ref{eg_5} and \ref{eg_6});

(3) Assumption \ref{assume_A} is tight for the full domain of $\cT$ (see Example \ref{eg_2});

(4) Assumption \ref{assume_A} is not necessary for the full domain of $\cT$ (see Example \ref{eg_3}).

\begin{example} \label{eg_1}
Consider a toy example
\[
\cH=\R^2,\quad  \cA = \partial f, \text{\ where\ }  f(x,y)=|x|, \quad 
\cQ=\begin{bmatrix} 0 &0 \\ 0& 1 \end{bmatrix}.
\]
It is easy to see that $\ransf\cA= [-1,1]\times \{0\}$, $\ransf\cQ = \{0\} \times \R$, and $\ransf(\cA+\cQ) = [-1,1]\times \R  \supset\ransf\cQ$, though $\cA+\cQ$ is not surjective. One further has $\ransf\cA-\ransf\cQ=[-1,1] \times \R$, which contains 0 as a relative interior. Thus, Corollary \ref{c_full} concludes that $(\sqrt{\cQ}\cA^{-1}\sqrt{\cQ}) \big|_{\ransf\cQ} $ is maximally monotone.

Let us verify it by checking whether $\ransf(\cI+\sqrt{\cQ}\cA^{-1}\sqrt{\cQ} \big|_{\ransf\cQ} )=\ransf\cQ$. Indeed, noting that 
$\cA^{-1}: (x,y)\mapsto \{0\} \times \R$, if $(x,y)\in (-1,1) \times \{0\}$, we develop:
$\ransf \big(\cI+ (\sqrt{\cQ}\cA^{-1}\sqrt{\cQ}) \big|_{\ransf\cQ} \big)    \supseteq    
\big(\cI+ (\sqrt{\cQ}\cA^{-1}\sqrt{\cQ}) \big|_{\ransf\cQ} \big) 
\begin{bmatrix}
 0 \\ 0 \end{bmatrix}
   =  \sqrt{\cQ} \cA^{-1} \begin{bmatrix}
0 \\ 0 \end{bmatrix}  =   \{0\} \times \R = \ransf\cQ$. This shows $(\sqrt{\cQ}\cA^{-1}\sqrt{\cQ}) \big|_{\ransf\cQ}$ is maximally monotone by Lemma \ref{l_minty}.
\end{example}

\begin{example} \label{eg_2}
In \cite[Remark 2.2]{bredies_ppa}, the authors proposed an interesting counter-example that deserves particular treatment. The problem is
\[
\cH =\R^2, \cA = \partial f, \text{\ where\ } f(x,y)=\max\{e^y-x, 0\} \text{\ and\ } \cQ = \begin{bmatrix}
1 &0 \\ 0&0 \end{bmatrix}.
\]
Based on basic subdifferential calculus,  $\cA$ and $\cA^{-1}$ are obtained as
\[
\cA: (x,y)\mapsto \left\{ \begin{array}{ll}
\{(-1, e^y)\}, & \textrm{if\ } e^y > x; \\
\{(0, 0)\}, & \textrm{if\ } e^y < x; \\ {}
 \{ (-t, te^y): t\in [0,1]\} , & \textrm{if\ } e^y = x.
\end{array}
\right.
\]
\[
\cA^{-1}: (u_1,u_2)\mapsto  
\left\{ \begin{array}{ll}
(-\infty, u_2) \times \{\log u_2\}, & \textrm{if\ }  (u_1,u_2) \in \{-1\} \times (0, +\infty); \\
\{ (-\frac{u_2}{u_1}, \log (-\frac{u_2}{u_1}) )\} , & \textrm{if\ }  (u_1,u_2)\in (-1, 0) \times  (0, +\infty); \\ { }
\{(p,q): p\ge e^q \}, & \textrm{if\ }  (u_1,u_2)=(0,0); \\
\varnothing, & \textrm{otherwise}.
\end{array}\right.
\]

We then have  $(\sqrt{\cQ}\cA^{-1}\sqrt{\cQ}) \big|_{\ransf\cQ}: 
(u_1, 0) \mapsto \left\{ \begin{array}{ll}
(0, +\infty) \times \{0\}, & \textrm{if\ }  u_1 = 0; \\
\varnothing, & \textrm{otherwise}.
\end{array}\right.$  It is obvious that $\ransf(\cA+\cQ)\mcap \ransf \cQ = (0,+\infty)\times \{0\}\subset \ransf\cQ=\R\times \{0\} $, which shows that $\ransf(\cA+\cQ)\nsupseteq \ransf \cQ$, and $(\cA+\cQ)^{-1}(x,0)=\varnothing$,  $\forall x \in (-\infty, 0]$.

Since $\cT$ is not of full domian, the monotonicity of $(\sqrt{\cQ}\cA^{-1}\sqrt{\cQ}) \big|_{\ransf\cQ}$  must not be maximal by  Theorem \ref{t_full}. Let us check it. Noting that $\domsf (\sqrt{\cQ}\cA^{-1}\sqrt{\cQ}) \big|_{\ransf\cQ} = \{ (0,0) \}$, we then have $\ransf \big(\cI + (\sqrt{\cQ}\cA^{-1}\sqrt{\cQ}) \big|_{\ransf\cQ} \big) =  \big(\cI + (\sqrt{\cQ}\cA^{-1}\sqrt{\cQ}) \big|_{\ransf\cQ} \big)  \begin{bmatrix}
0 \\ 0 \end{bmatrix} = \sqrt{\cQ}\cA^{-1} \begin{bmatrix}
0 \\ 0 \end{bmatrix} 
 =  (0, +\infty) \times \{0\} \subset \R \times\{ 0\}$.  Thus, by Lemma \ref{l_minty},  $(\sqrt{\cQ}\cA^{-1}\sqrt{\cQ}) \big|_{\ransf\cQ}$ is not maximally monotone.

Finally, it is easy to check that Assumption \ref{assume_A}-(ii) is not satisfied. Indeed,$\ransf\cA=[-1,0)\times [0,+\infty) \mcup (0,0) $  and $\ransf\cQ -\ransf\cA=\R \times (-\infty, 0]$, and thus, $0\notin \srisf(\ransf\cA - \ransf\cQ)$. This shows that Assumption \ref{assume_A} is tight. In addition, observe that $0\in \ransf\cA$ and $\overline{\ransf}(\cA+\cQ) = \R\times [0,+\infty) \supset  \ransf\cQ$, which verifies Proposition \ref{p_close}.
\end{example}

\begin{example} \label{eg_3}
We consider Example \ref{eg_2} but with the metric  $\cQ=\begin{bmatrix}
0 &0\\0&1 \end{bmatrix}$. Then, based on $\cA$ and $\cA^{-1}$ given in Example \ref{eg_2},  we deduce that
$(\sqrt{\cQ}\cA^{-1}\sqrt{\cQ}) \big|_{\ransf\cQ}: 
(0, u_2) \mapsto \left\{ \begin{array}{ll}
\{0\} \times \R, & \textrm{if\ }  u_2 = 0; \\
\varnothing, & \textrm{otherwise}.
\end{array}\right.$
 $\ransf(\cA+\cQ)=[-1,0]\times \R$. Clearly, $\ransf(\cA+\cQ)\supseteq \ransf\cQ= \{0\} \times \R $. By Theorem \ref{t_full},  $(\sqrt{\cQ}\cA^{-1}\sqrt{\cQ}) \big|_{\ransf\cQ}$  must be maximally monotone. Indeed, $\ransf(\cI + \sqrt{\cQ}\cA^{-1}\sqrt{\cQ} \big|_{\ransf\cQ}) = 
(\cI + \sqrt{\cQ}\cA^{-1}\sqrt{\cQ} \big|_{\ransf\cQ}) 
\begin{bmatrix}
0 \\ 0 \end{bmatrix} =  \sqrt{\cQ}\cA^{-1}
\begin{bmatrix} 0 \\ 0 \end{bmatrix}
=\{0\}\times \R =\ransf\cQ$, and thus by Lemma \ref{l_minty},  $(\sqrt{\cQ}\cA^{-1}\sqrt{\cQ}) \big|_{\ransf\cQ}$ is  maximally monotone.

However, $\ransf\cQ -\ransf\cA=[0,1] \times \R$ and   $0\notin \srisf(\ransf\cA - \ransf\cQ)$.  This shows that the condition of  $0\in \srisf(\ransf\cA - \ransf\cQ)$ is sufficient but not necessary condition for the full domain of $\cT$. 
\end{example}

The following examples show that  $(\sqrt{\cQ}\cA^{-1}\sqrt{\cQ}) \big|_{\ransf\cQ}$ may be  maximally monotone, though $\cA$ is not.
\begin{example} \label{eg_4}
Considering $\cA:\R^2\mapsto\R^2: x \mapsto
\left\{ \begin{array}{ll}
x, & \textrm{if\ }  x\ne (0, -1); \\
\varnothing, & \textrm{otherwise}.
\end{array}\right.$ and $\cQ = \begin{bmatrix}
1 & 0 \\ 0 &0 \end{bmatrix}$, $\cA$ is obviously not maximally monotone, because $\grasf\cA$ can be extended by adding back the lacking point $\{ ( (0,-1), (0,-1) )\}$ to recover the identity operator. However, $(\sqrt{\cQ}\cA^{-1}\sqrt{\cQ}) \big|_{\ransf\cQ} = \cI|_{\ransf\cQ}: (x_1, 0)\mapsto (x_1, 0)$, $\forall x_1\in\R$. This is linear and continuous, and hence maximally monotone \cite[Corollary 20.28]{plc_book}.
\end{example}

One can also restrict the domain of  $\cA$ in Example \ref{eg_4} to a proper subspace as follows.
\begin{example} \label{eg_5}
Considering $\cA:\R^2\mapsto\R^2: x \mapsto
\left\{ \begin{array}{ll}
x, & \textrm{if\ }  x\in \R\times \{0\}; \\
\varnothing, & \textrm{otherwise}.
\end{array}\right.$ and $\cQ = \begin{bmatrix}
1 & 0 \\ 0 &0 \end{bmatrix}$, $\cA$ is obviously not maximally monotone, because $\grasf\cA$ can be extended by, for example, adding the point $\{ ( (0,1), (0,1) )\}$.  However, $(\sqrt{\cQ}\cA^{-1}\sqrt{\cQ}) \big|_{\ransf\cQ} = \cI|_{\ransf\cQ}: (x_1, 0)\mapsto (x_1, 0)$ ($\forall x_1\in\R$) is maximally monotone.
\end{example}

Example \ref{eg_5} also shows that it is difficult to enforce maximality to the monotonicity of an operator defined on a proper subspace, as discussed in Proposition \ref{p_fail}. The next example connects $\cA$ to the subdifferential of some function.
\begin{example} \label{eg_6}
Considering $f:\R^2\mapsto \R: (x_1,x_2) \mapsto 
\left\{ \begin{array}{ll}
\frac{1}{2} x_1^2, & \textrm{if\ }  x\ne (0,1); \\
1, & \textrm{if\ }  x= (0,1)
\end{array}\right.$, and thus 
 $ \partial f:  (x_1,x_2) \mapsto
\left\{ \begin{array}{ll}
(x_1, 0) , & \textrm{if\ }  x\ne (0,1); \\
\varnothing, & \textrm{otherwise}.
\end{array}\right.$. This is not maximally monotone, because $\grasf\cA$ can be extended by adding the point $\{ ( (0,1), (0,0) )\}$.  
Noting  $ (\partial f)^{-1}:  (x_1,x_2) 
\mapsto \left\{ \begin{array}{ll}
\{0\} \times (\R \backslash \{1\}), 
& \textrm{if\ }  (x_1,x_2)=(0, 0); \\
\{x_1\} \times \R, & \textrm{if\ }   (x_1, x_2) \in (\R\backslash \{0\} ) \times \{0\}; \\
\varnothing, & \textrm{if\ } (x_1, x_2) \in \R\times (\R\backslash \{0\} )
\end{array}\right.$ and letting $\cQ = \begin{bmatrix}
1 & 0 \\ 0 &0 \end{bmatrix}$, we obtain $(\sqrt{\cQ} \circ (\partial f)^{-1} \circ \sqrt{\cQ}) \big|_{\ransf\cQ} = \cI|_{\ransf\cQ}$, which is maximally monotone.
\end{example}

\section{Single-valuedness of degenerate preconditioned resolvent} \label{sec_single}
Let us now study the condition for the  single-valuedness of $\cT$. More precisely, $\cT$ given as \eqref{T_single} is said to be single-valued, if $\cT x$ is always a singleton $\{y\}$, $\forall x\in\domsf \cT$.

\subsection{Restricted injectivity}
As mentioned in Sect. \ref{sec_related}, a standard condition for the single-valuedness of $\cT$ is  disjoint injectivity of $\cA+\cQ$  \cite[Theorem 2.1-(ix)]{bau_review}, \cite[Proposition 3.8]{plc_warped}. Moreover, a sufficient and necessary condition, proposed in \cite[Proposition 1-(2)]{arias_infimal} and \cite[Proposition 2.4]{naldi_thesis}, is that $\cA+\cQ$ is injective on $\ransf\cQ$. Here, Theorem \ref{t_single} claims that this condition can be further weakened to  the {\it disjoint injectivity on $\ransf\cQ$ w.r.t. $\kersf\cQ$} only, under Assumptions \ref{assume_Q} and   \ref{assume_local}. 

\begin{definition} \label{def_inj}
$\cA+\cQ$ is disjointly injective on $\ransf\cQ$ w.r.t. $\kersf\cQ$, i.e., it satisfies  
\be \label{e2}
(\cA+\cQ) y_1 \mcap (\cA+\cQ) y_2 \mcap \ransf\cQ \ne \varnothing
\Longrightarrow y_{1,\ksf} = y_{2,\ksf}, \quad
\forall y_1,y_2\in \domsf (\cA+\cQ).
\ee
\end{definition}

For convenience of reference, we put Definition \ref{def_inj} as the following assumption.
\begin{assumption} \label{assume_single}
{\rm [Single-valuedness of $\cT$]}

{\rm (i)} $\cA+\cQ$ is disjointly injective on $\ransf\cQ$ w.r.t. $\kersf\cQ$;

{\rm (ii)} $\kersf\cQ \mcap \big( \cA^{-1}\cQ(x_\rsf-y_\rsf) - y_\rsf \big)$  is a singleton, $\forall x\in\domsf \cT$, where $y_\rsf = (\cP_{\ransf\cQ} \circ \cT) (x_\rsf)$;

{\rm (iii)} $\kersf\cQ \mcap \big( \cA^{-1}\cQ(x_\rsf-y_\rsf) - y \big) = \{ 0 \}$, $\forall x\in\domsf \cT$, provided that a particular element $y\in\cT x$ is given;

{\rm (iv)}  $\cP_{\kersf\cQ} \circ \cA^{-1} \circ \cQ$ is single-valued.
\end{assumption}

Proposition \ref{p_single} shows the relations of Assumption \ref{assume_single}.
\begin{proposition} \label{p_single}
Regarding \eqref{T_single} under Assumptions \ref{assume_Q} and \ref{assume_local}, Assumption \ref{assume_single} has the following  relations: (iv)$\Longrightarrow$(i)$\Longleftrightarrow$(ii)$\Longleftrightarrow$(iii).
\end{proposition}
\begin{proof}
The proof follows the order of (i)$\Longrightarrow$(ii)$\Longrightarrow$(iii)$\Longrightarrow$(i).

 (i)$\Longrightarrow$(ii): Since $(\cA+\cQ)y_1 \mcap (\cA+\cQ)y_2 \mcap \ransf\cQ \ne \varnothing$, pick  a point $x$, such that 
$\cQ x \in (\cA+\cQ)y_1 \mcap (\cA+\cQ)y_2$.
We have by (i) that $y_{1,\ksf} = y_{2,\ksf} :=y_\ksf$. This also implies that $\{y_1,y_2\}\subseteq (\cA+\cQ)^{-1}\cQ x = \cT x \Longrightarrow y_{1,\ksf} = y_{2,\ksf}$. Then (ii) is obtained, combining with the proof of Lemma \ref{l_X}.

(ii)$\Longrightarrow$(iii): Let $\kersf\cQ \mcap \big( \cA^{-1}\cQ(x_\rsf-y_\rsf) - y_\rsf \big) = \{y_\ksf\}$. Subtracting  $y_k$ on both sides yields: $ (\kersf\cQ-y_\ksf) \mcap \big( \cA^{-1}\cQ(x_\rsf-y_\rsf) - y_\rsf - y_\ksf\big) 
= \kersf\cQ \mcap \big( \cA^{-1}\cQ(x_\rsf-y_\rsf) - y_\rsf - y_\ksf\big) 
= \kersf\cQ \mcap \big( \cA^{-1}\cQ(x_\rsf-y_\rsf) - y \big) 
= \{0\}$, where the first equality is due to $y_\ksf\in \kersf \cQ$ and $\kersf\cQ-y_\ksf = \kersf \cQ$, the second equality comes from $y=y_\rsf+y_\ksf$, by $y\in\cT x$ and Lemma \ref{l_X}.

(iii)$\Longrightarrow$(i):  By the equivalence between (iii) and (ii), and Lemma \ref{l_X}, we have by (iii) that $\{y_1,y_2\}\subseteq \cT x \Longrightarrow y_{1,\ksf} = y_{2,\ksf}$. Noting that $\cT = (\cA+\cQ)^{-1}\cQ$, this is equivalent to $\cQ x\in (\cA+\cQ) y_1 \mcap (\cA+\cQ) y_2 \Longrightarrow  y_{1,\ksf} = y_{2,\ksf}$, which is exactly (i). 

Finally, we show (iv)$\Longrightarrow$(i).  Since $(\cA+\cQ)y_1 \mcap (\cA+\cQ)y_2 \mcap \ransf\cQ \ne \varnothing$, pick  a point $x$, such that 
$\cQ x \in (\cA+\cQ)y_1 \mcap (\cA+\cQ)y_2$. By Lemma \ref{l_single}, we have $y_{1,\rsf} = y_{2,\rsf} :=y_\rsf$, and $y_{1,\ksf}+y_\rsf \in \cA^{-1} \cQ (x_\rsf-y_\rsf)$,  $y_{2,\ksf}+y_\rsf \in \cA^{-1} \cQ (x_\rsf-y_\rsf)$.  Then, apply the projection $\cP_{\kersf\cQ}$ to both sides yields
$\{ y_{1, \ksf},  y_{2, \ksf} \} \subseteq \cP_{\kersf\cQ} \cA^{-1} \cQ (x_\rsf-y_\rsf)$. By (iv), we have $y_{1, \ksf} = y_{2, \ksf} $, from which follows (i).
\end{proof}

From Proposition \ref{p_single} immediately follows the next theorem.
\begin{theorem} \label{t_single}
Given \eqref{T_single} under Assumptions \ref{assume_Q} and  \ref{assume_local}, $\cT$ is single-valued, 

{\rm (i)} if and only if one of Assumption \ref{assume_single}-(i--iii)  holds;

{\rm (ii)} if  Assumption \ref{assume_single}-(iv) holds.
\end{theorem}
\begin{proof}
(i) Clear by the proof of Proposition \ref{p_single}. It is also easy to see from the proof of Lemma \ref{l_X} the equivalence between single-valuedness and Proposition \ref{p_single}-(ii).

(ii) Clear by the proof of Proposition \ref{p_single}.
\end{proof}

\begin{remark}
{\rm (i)} Theorem \ref{t_single}-(ii) is sufficient  but not necessary for single-valuedness. Actually, this condition is much stronger than (i), which will be shown in Sect. \ref{sec_eg_2}.

{\rm (ii)}  There is no need to enforce $\cP_{\ransf\cQ} \cA^{-1}\cQ$ to be single-valued, since the uniqueness of $y_\rsf$ has been guaranteed by Lemma \ref{l_single}. This is irrelevant to single-valuedness of $\cP_{\ransf\cQ}\cA^{-1}\cQ$, which is, indeed, often multi-valued in practice.

{\rm (iii)} Assumption \ref{assume_single}-(iii) has similar spirit with the most recent result of \cite[Theorem 3.1]{fadili_cone} for regularized optimization problems. Their potential connections, particularly from the geometric view, deserve further explorations in future work.
\end{remark}

We are now ready to summarize the conditions for the well-definedness of $\cT$.
\begin{corollary} \label{c_well}
Under Assumptions \ref{assume_Q} and \ref{assume_local}, 
$\cT$ given as \eqref{T}  is well-defined,

{\rm (i)}  if and only if Assumptions \ref{assume_max} and \ref{assume_single}-(i) (or (ii), (iii)) are fulfilled;

{\rm (ii)}  if Assumptions \ref{assume_A} and \ref{assume_single}-(iv) are satisfied.
\end{corollary}

\subsection{A guiding example} \label{sec_eg_2}
Consider the following example of \eqref{T_single} with
\[
\cA = \partial f\textrm{\ with\ } f: \R^3\mapsto\R: x\mapsto \|x\|_1,\quad \cQ=\begin{bmatrix}
1 &0&2 \\  0&4&-4 \\ 2& -4& 8 \end{bmatrix},\quad
x=\begin{bmatrix} 0 \\  1\\  1/2 \end{bmatrix}. 
\]  
It is easy to verify that  $\cT x$ is a singleton $\{y\}$ with   $y =\begin{bmatrix}
0  & 1/4  & 0 \end{bmatrix}^\top$, and thus, 
$q = \cQ (x_\rsf - y_\rsf) =  \begin{bmatrix}
1 &  1 &  1 \end{bmatrix}^\top$, and
$\cA^{-1} \cQ (x_\rsf - y_\rsf)= \partial (\|\cdot\|_1^*) (q) =[0,+\infty)\times [0,+\infty)\times [0,+\infty)$. Using  $\kersf \cQ = \xi \begin{bmatrix}
-2 &  1 & 1 \end{bmatrix}^\top$, $\forall \xi\in\R$, the single-valuedness coincides with $\partial (\|\cdot\|_1^*) (q) \mcap \kersf \cQ = \{0\}$, but $\cP_{\kersf \cQ} \big( \partial  (\|\cdot\|_1^*) (q)
\big) $, representing the projection of positive orthant in $\R^3$ onto the line $\R(-2, 1, 1)$, is clearly multi-valued. This shows that  Assumption \ref{assume_single}-(iv) is not necessary for single-valuedness of $\cT$.


\section{Convergence analysis}
\label{sec_con}
From now on, we start to assume the solution existence of \eqref{p1}, which is not needed in the previous sections.
\begin{assumption} \label{assume_zero}
{\rm [Solution existence]} 
$\zersf \cA := \cA^{-1} (0) \ne \varnothing$.
\end{assumption}

Assumption \ref{assume_zero} is quite a standard condition in literature, though it would be also interesting to investigate the behaviours of PPA \eqref{T} for the case of $\zersf\cA =\varnothing$, which is known as {\it inconsistent feasibility setting} \cite{hhb_drs}.

\subsection{Convergence in the range space}
Now, we turn to the convergence analysis of \eqref{T}.  The following is an intermediate result from  the proof of \cite[Theorem 3.4]{latafat_2017}, which is also crucial  for our convergence analysis. 
\begin{lemma} \label{l_fix}
Given $\cT$ as \eqref{T} under Assumptions \ref{assume_Q} and \ref{assume_local}, then  
\[
\Fixsf (\cP_{\ransf\cQ} \circ \cT) = 
\cP_{\ransf\cQ} (\zersf \cA).
\]
\end{lemma}
\begin{proof}
(1) $x\in \zersf\cA \Longrightarrow 0 \in \cA x \Longrightarrow \cQ x \in (\cA+\cQ) x \Longrightarrow x \in (\cA+\cQ)^{-1}\cQ x \Longrightarrow \cP_{\ransf\cQ} x = \cP_{\ransf\cQ} (\cA+\cQ)^{-1}\cQ x =( \cP_{\ransf\cQ} \circ \cT)( x)  
 = (\cP_{\ransf\cQ} \circ \cT)( \cP_{\ransf\cQ} x) \Longrightarrow
 \cP_{\ransf\cQ} x \in \Fixsf(\cP_{\ransf\cQ} \circ \cT)$.
Therefore, $x\in \zersf\cA \Longrightarrow \cP_{\ransf\cQ} x\in \cP_{\ransf\cQ} (\zersf\cA)$. 
 Thus, we obtain  $\cP_{\ransf\cQ} (\zersf\cA) \subseteq \Fixsf(\cP_{\ransf\cQ} \circ \cT)$.

\vskip.2cm
(2) Let $x \in \Fixsf(\cP_{\ransf\cQ} \circ \cT) \Longrightarrow
x = (\cP_{\ransf\cQ} \circ \cT) x \Longrightarrow
\cQ x = \cQ (\cP_{\ransf\cQ} \circ \cT) x = \cQ \cT x$.
On the other hand, $\cT = (\cA+\cQ)^{-1}\cQ \Longrightarrow  0\in\cA (\cT x) +\cQ \cT x-\cQ x$, $\forall x\in \domsf\cT$. Combining both above, we then have $x \in \Fixsf(\cP_{\ransf\cQ} \circ \cT) \Longrightarrow 
0 \in\cA (\cT x) \Longrightarrow \cT x \subseteq \zersf\cA \Longrightarrow (\cP_{\ransf\cQ} \circ \cT) (x) \in \cP_{\ransf\cQ} (\zersf\cA)$. Since $x = (\cP_{\ransf\cQ} \circ \cT) (x)$, we have $ x \in \cP_{\ransf\cQ} (\zersf\cA)$. Thus, $\Fixsf(\cP_{\ransf\cQ} \circ \cT) \subseteq \cP_{\ransf\cQ} (\zersf\cA)$.
\end{proof}

Theorem \ref{t_con_ran} establishes the weak convergence of $\cP_{\ransf\cQ} x^k \weak \cP_{\ransf\cQ} x^\star$, and its connections to existing works will be discussed in Sect. \ref{sec_diss}. 
For notational simplicity, we  denote $x_\rsf^k :=\cP_{\ransf\cQ} x^k$, $\forall k\in\N$.
\begin{theorem}    \label{t_con_ran} 
{\rm [Weak convergence in $\ran \cQ$]}
Under Assumptions \ref{assume_Q}, \ref{assume_local}, \ref{assume_max} and \ref{assume_zero}, let $\{x^k\}_{k\in \N}$ be a sequence generated by \eqref{T}. Then $x_\rsf^k \weak \cP_{\ransf\cQ} x^\star$ for some $x^\star \in \zersf \cA$, as $k \rightarrow \infty$.
\end{theorem}

\vskip.2cm
\begin{proof}
The proof is divided into 3 steps:

Step-1: For every $x_\rsf^\star \in \cP_{\ransf\cQ} ( \zersf \cA)$, $\lim_{k\rightarrow \infty}$  $\| x_\rsf^k - x_\rsf^\star \|_\cQ $ exists.  

Step-2: $\{ x_\rsf^k\}_{k\in\N}$ has at least  one weak sequential cluster point lying in $ \cP_{\ransf\cQ} (\zersf\cA)$;

Step-3: The weak sequential  cluster point of  $\{x_\rsf^k\}_{k\in\N}$ is unique.

\vskip.2cm
Step-1: First, we define $x_\rsf^\star \in \Fixsf(\cP_{\ransf\cQ} \circ \cT)$. We also have  $x_\rsf^\star \in \cP_{\ransf\cQ} (\zersf\cA)$  by Lemma \ref{l_fix}, which implies that $\exists x^\star \in \zersf\cA$, such that $x_\rsf^\star = \cP_{\ransf\cQ} x^\star$.  According to Fact \ref{f_fne}, substituting $x=x_\rsf^k$  and $y=x_\rsf^\star$ into \eqref{fne_r} yields 
\be  \label{av}
 \big\| x_\rsf^{k+1} -  x_\rsf^\star \big\|_\cQ^2 
\le   \big\| x_\rsf^k - x_\rsf^\star \big\|_\cQ^2
-   \big \| x_\rsf^k - x_\rsf^{k+1}  \|_\cQ^2.
\ee
This guarantees that  $\{ \| x_\rsf^{k} - x_\rsf^\star \|_\cQ \}_{k\in\N} $ is non-increasing, and bounded from below (always being non-negative), and thus, convergent, i.e., $\lim_{k\rightarrow \infty}$  $\| x_\rsf^k - x_\rsf^\star \|_\cQ $ exists.  
 
\vskip.2cm
Step-2: \eqref{av} implies the boundedness of $\{ x_\rsf^k\}_{k\in\N}$. Indeed, from \eqref{av} and the closedness of $\ransf\cQ$, according to Remark \ref{rmk_closed},  there exists $\alpha>0$, such that
\[
\alpha \big\| x_\rsf^k  \big\| 
-\alpha \big\|  x_\rsf^\star \big\| 
\le \alpha \big\| x_\rsf^k - x_\rsf^\star \big\| 
\le \big\| x_\rsf^k - x_\rsf^\star \big\|_\cQ
\le   \big \| x_\rsf^0 - x_\rsf^\star  \|_\cQ,
\quad \forall k\in\N.
\]
Thus, $ \big\| x_\rsf^k  \big\| 
\le   \big\|  x_\rsf^\star \big\| 
+\frac{1}{\alpha} 
  \big \| x_\rsf^0 - x_\rsf^\star  \|_\cQ :=\rho$, $\forall k\in\N$, and $\{ x_\rsf^k\}_{k\in\N} \subseteq B(0,\rho)\cap \ransf\cQ$.
$B(0,\rho)$ is weakly sequentially compact (by \cite[Lemma 2.45]{plc_book}), and also is  $ B(0,\rho)\cap \ransf\cQ \subseteq B(0,\rho)$ (by \cite[Lemma 1.34]{plc_book}). Then, $\{ x_\rsf^k\}_{k\in\N} $ must possess at least one weak sequential cluster point $v^* \in B(0,\rho)\cap \ransf\cQ $,  i.e.,   there exists a subsequence $\{ x_\rsf^{k_i}\}_{k\in\N}$ that weakly converges to $v^*$, as $k_i\rightarrow \infty$. We here need to show that $v^* \in  \cP_{\ransf\cQ} (\zersf\cA)$, and more generally,  every weak sequential cluster point of  $\{ x_\rsf^k\}_{k\in\N}$ belongs to $ \cP_{\ransf\cQ} (\zersf\cA)$. To this end, summing up \eqref{av} from $k=0$ to $K-1$, and taking $K \rightarrow \infty$, we have
\[
\sum_{k=0}^{\infty} \big\| x_\rsf^k - x_\rsf^{k+1} \big\|_\cQ^2  
\le \big\| x_\rsf^{0} - x_\rsf^\star \big\|_\cQ^2 < +\infty.
\]
This implies that  $ x_\rsf^k-  x_\rsf^{k+1} = (\cI - \cP_{\ransf\cQ} \circ \cT) (x_\rsf^{k} ) 
\rightarrow  0$, and also $ (\cI - \cP_{\ransf\cQ} \circ \cT) (x_\rsf^{k_i} ) \rightarrow  0$ as $k_i \rightarrow \infty$. Combining with $x_\rsf^{k_i} \weak v^* \in \ransf\cQ$ and the $\cQ$-demiclosedness of $\cI- \cP_{\ransf\cQ} \circ \cT|_{\ransf\cQ}$ (shown in Lemma \ref{l_demi_T}), we obtain $ (\cI - \cP_{\ransf\cQ} \circ \cT|_{\ransf\cQ}) (v^*) = (\cI - \cP_{\ransf\cQ} \circ \cT) (v^*)= 0 \Longrightarrow v^* \in \Fixsf ( \cP_{\ransf\cQ} \circ \cT)$. Since $\{ x_\rsf^{k_i}\}_{k_i\in\N}$ is an arbitrary weakly convergent subsequence of $\{x_\rsf^k\}_{k\in\N}$, we conclude that every weak sequential cluster point of $\{x_\rsf^k\}_{k\in\N}$ lies in $\Fixsf ( \cP_{\ransf\cQ} \circ \cT) $, and furthermore, in  $\cP_{\ransf\cQ} (\zersf\cA)$ as well, again by Lemma \ref{l_fix}.


\vskip.2cm
Step-3: We need to show that   $\{x_\rsf^k\}_{k\in\N}$ cannot have two distinct weak sequential cluster point in $ \cP_{\ransf\cQ} (\zersf\cA)$. Indeed, let $v_1^*, v_2^{*} \in  \cP_{\ransf\cQ} (\zersf\cA)$ be two cluster points of  $\{  x_\rsf^k\}_{k\in\N}$. Since  $\lim_{k\rightarrow \infty} \| x_\rsf^{k} - x_\rsf^\star \|_\cQ $ exists as proved in Step-1, set $l_1 = \lim_{k\rightarrow \infty} \| x_\rsf^k - v_1^*\|_\cQ $, and $l_2 = \lim_{k\rightarrow \infty} \| x_\rsf^k - v_2^*\|_\cQ $. Take a subsequence $\{ x_\rsf^{k_i} \}_{k_i\in\N}$ weakly converging to $v_1^*$, as $k_i \rightarrow \infty$. From the identity:
\[
\big\| x_\rsf^{k_i}-v_1^*\big\|_\cQ^2 - \big\|x_\rsf^{k_i}-v_2^* \big\|_\cQ^2
=\big\|v_1^*- v_2^*\big\|_\cQ^2 +2 \big\langle 
v_1^*- v_2^* \big|  v_2^* - x_\rsf^{k_i} \big\rangle_\cQ,
\]
taking $k_i\rightarrow \infty$ on both sides, the last term becomes:
\[
 \lim_{k_i \rightarrow \infty} \big\langle 
v_1^*- v_2^* \big|  v_2^* - x_\rsf^{k_i} \big\rangle_\cQ
= \lim_{k_i \rightarrow \infty} \big\langle 
v_1^*- v_2^* \big|  v_2^* -  x_\rsf^{k_i} \big\rangle_\cQ
= - \|v_1^*- v_2^*\|_\cQ^2,
\]
then, we deduce that  $l_1^2 - l_2^2 =- \big\| v_1^*- v_2^*\big\|_\cQ^2$. Similarly,  take a subsequence $\{  x^{k_j} \}_{k_j \in\N}$ weakly converging to $ v_2^*$, as $k_j \rightarrow \infty$, which yields that $l_1^2 - l_2^2 = \big\|v_1^* - v_2^*\big\|_\cQ^2$.  Consequently, $\big\|v_1^*-v_2^*\big\|_\cQ = 0 \Longrightarrow  v_1^* = v_2^*$, since $v_1^*,v_2^* \in \ransf\cQ$. This shows the uniqueness of the weak sequential cluster point.

Finally, combining the above 3 step and \cite[Lemma 2.46]{plc_book}, we conclude the weak convergence of $\{x_\rsf^k\}_{k\in\N}$, and denote the weak limit as $x_\rsf^\star \in \cP_{\ransf\cQ} (\zersf\cA)$.
\end{proof}

\subsection{Discussions and related works}
\label{sec_diss}
A number of technical details in the proof of Theorem \ref{t_con_ran}  need to be clarified, discussed and further compared to other related works.

First, Lemma \ref{l_fix}  makes the connection between the degenerate PPA \eqref{T} and the original inclusion problem \eqref{p1}. One can always show the weak convergence of  $\{x_\rsf^k\}_{k\in\N}$ in $\Fixsf(\cP_{\ransf\cQ} \circ \cT)$ without Lemma \ref{l_fix}. This is Lemma \ref{l_fix} that tells us what the degenerate PPA \eqref{T} solves for is actually a point of $\zersf\cA$ projected onto the range space of degenerate preconditioner. 

Assumption \ref{assume_A}, which is stronger than Assumption \ref{assume_max}, is not needed in Theorem \ref{t_con_ran}, and thus, it is not necessary to assume  $\cA:\cH\mapsto 2^\cH$ to be (globally) maximally monotone here. Moreover, Theorem \ref{t_single} is also not needed. In other words, $\cT$ is allowed to be multi-valued, since we are only concerned with $\{x_\rsf^k\}_{k\in\N}$ here, which is already well-defined by Assumptions \ref{assume_Q}, \ref{assume_max} and Lemma \ref{l_single}. This is fundamentally different from \cite{latafat_2017,bredies_2017,bredies_ppa}, where $\cT$ is required to be well-defined. Theorem \ref{t_con_ran} essentially proves the weak convergence of the well-defined iteration $x_\rsf^{k+1} = (\cP_{\ransf\cQ} \circ \cT) (x_\rsf^k)$.

\cite[Theorem 3.4]{latafat_2017} discussed the convergence of $\{ x_\rsf^k\}_{k\in\N}$ in the finite-dimensional setting, where the proof heavily relied on the assumption of continuity of $\cT$. We here argue that  the continuity of $\cP_{\ransf\cQ} \circ \cT|_{\ransf\cQ}:\ransf\cQ \mapsto 2^{\ransf\cQ}$ would suffice for \cite[Theorem 3.4]{latafat_2017}, and this is obviously true in finite-dimensional case by Fact \ref{f_fne}.  However, $\cP_{\ransf\cQ} \circ \cT|_{\ransf\cQ}$ is not weakly continuous generally in infinite-dimensional case. Instead, we exploited the weak-to-strong continuity of $\cI - \cP_{\ransf\cQ} \circ \cT|_{\ransf\cQ}$ in Lemma \ref{l_demi_T}.

The Opial's lemma (see \cite[Lemma 2.1]{attouch_2001} for example) or \cite[Lemma 2.47]{plc_book} can also be used in the proof instead of \cite[Lemma 2.46]{plc_book}, without Step-3.
There are also other ways to finish the proof. For instance, notice that the sequence $\{x_\rsf^k\}_{k\in\N}$  we are concerned with here lies in $\ransf\cQ$ only. If we treat $\ransf\cQ$ as the ambient space equipped with $\|\cdot\|_\cQ$ (strong topology) and $\langle \cdot |\cdot \rangle_\cQ$ (weak topology)\footnote{$(\ransf\cQ, \|\cdot\|_\cQ)$ and   $(\ransf\cQ, \langle \cdot |\cdot \rangle_\cQ)$ are Hausdorff spaces, and however, $(\cH, \|\cdot\|_\cQ)$ and   $(\cH, \langle \cdot |\cdot \rangle_\cQ)$ are not.}, the weak convergence of   $\{x_\rsf^k\}_{k\in\N}$ in $\cP_{\ransf\cQ}(\zersf\cA)$ can be obtained following a standard analysis of \cite{attouch_2001,ppa_guler,plc}, based on the nice properties of $\cP_{\ransf\cQ}\circ \cT|_{\ransf \cQ}$ in Sect. \ref{sec_local}.  More broadly,  many  concepts, e.g., weak (sequential) compactness, can be slightly modified to the metric space $(\ransf\cQ, \|\cdot\|_\cQ)$ (strong topology) and  $(\ransf\cQ, \langle \cdot| \cdot\rangle_\cQ)$ (weak topology). For example, one can define a set $C\subset\ransf\cQ$ as  {\it weakly sequentially compact w.r.t. $\langle\cdot| \cdot\rangle_\cQ$}, if  every sequence in  $C$ has a weak sequential cluster point in $C$ in a sense of $\langle\cdot| \cdot\rangle_\cQ$. In other words, for every sequece $x_\rsf^k \subset C$, there is a subsequence $\{x_\rsf^{k_i}\}_{k_i\in\N}$, such that $x_\rsf^{k_i} \weak x_\rsf^\star \in C$ in a sense of $\langle\cdot| \cdot\rangle_\cQ$, i.e., $\langle x_\rsf^{k_i} | u \rangle_\cQ \rightarrow 
\langle x_\rsf^\star  | u \rangle_\cQ $, $\forall u\in\ransf\cQ$. In this way, it is also easy to obtain the weak convergence of $\{x_\rsf^{k}\}_{k \in\N}$.

Our proof uses the $\cQ$-demiclosedness of $\cI - \cP_{\ransf\cQ} \circ \cT|_{\ransf\cQ}$. An alternative way is to use the maximal monotonicity of  $(\sqrt{\cQ}\cA^{-1}\sqrt{\cQ}) \big|_{\ransf\cQ}$ to finish the proof as well, see Appendix \ref{app_1}.

Finally, we stress that  $\{x^k\}_{k\in\N}$ may be unbounded due to the uncontrolled component of $x_\ksf^k := \cP_{\kersf\cQ} x^k$. Thus, $\{x^k\}_{k\in\N}$ may not even possess any weak sequential cluster point, and there is no conclusion about the convergence of   $\{x^k\}_{k\in\N}$ for the moment.  The only exception occurs in the non-degenerate case, where  the weak convergence of $\{x_\rsf^k\}_{k\in\N}$ is obviously equivalent to  that of $\{x^k\}_{k\in\N}$. 
Thus, Theorem \ref{t_con_ran} also covers the standard non-degenerate setting.

\subsection{Convergence in the whole space}
To proceed further, if $\cT$ is  well-defined, it is then natural to ask when the iterations \eqref{T} converge in the whole space. At first sight, the boundedness of $\{x^k\}_{k\in\N}$ should be a basic requirement. To this end, \cite[Theorem 2.9]{bredies_ppa} assumed the Lipschitz continuity of $(\cA+\cQ)^{-1}$. Here, we propose the following equivalent conditions, which further weakens \cite[Theorem 2.9]{bredies_ppa} to the Lipschitz continuity of $(\cA+\cQ)^{-1}|_{\ransf\cQ}$.

\begin{assumption} \label{assume_lip}
{\rm [Lipschitz continuity]}

{\rm (i)}  $(\cA+\cQ)^{-1}|_{\ransf\cQ}$ is Lipschitz continuous; 

{\rm (ii)}  $\cP_{\kersf\cQ} \circ \cT|_{\ransf\cQ}$ is Lipschitz continuous; 

{\rm (iii)} There exists a constant $\xi>0$, such that $\| \cT x_1-\cT x_2\| \le \xi\| x_1- x_2\|_\cQ$, $\forall x_1,x_2\in \cH$.
\end{assumption}

Assumption \ref{assume_lip}-(iii) can also be found in \cite[Theorem 3.3]{fxue_rima} and \cite[Lemma 3.2]{bredies_2017}.
Note that the prerequisite of Assumption \ref{assume_lip} is that $\cT$ should be well-defined, and thus, it can be used only when Assumptions \ref{assume_Q}, \ref{assume_local}, \ref{assume_max} and \ref{assume_single}-(i) (or (ii), (iii)) are fulfilled. The following proposition states that the three items of Assumption \ref{assume_lip} are equivalent.
\begin{proposition} \label{p_lip}
Given the well-defined $\cT$ as \eqref{T_single}, the three conditions of Assumption \ref{assume_lip} are equivalent.
\end{proposition}

The  proof is postponed in Appendix \ref{app_2}. Then, the convergence in the whole space immediately follows.
\begin{theorem} \label{t_con}
Let $\{x^k\}_{k\in\N}$ be a sequence  generated by \eqref{T}, which satisfies Assumptions \ref{assume_Q}, \ref{assume_local}, \ref{assume_max}, \ref{assume_zero} and any condition of Assumptions \ref{assume_single} and \ref{assume_lip}. Furthermore, if any of the following conditions holds:

{\rm (i)} every weak sequential cluster point of $\{x^k\}_{k\in\N}$ lies in $\Fixsf\cT$; 

{\rm (ii)} Assumption \ref{assume_A};

{\rm (iii)}  $\cI - \cT$ is demiclosed,

\noindent 
then $\{x^k\}_{k\in\N}$ weakly converges to some $x^\star \in \zersf\cA$, as $k\rightarrow \infty$.
\end{theorem}
\begin{proof}
Assumption \ref{assume_lip} guarantees the boundedness of  $\{x^k\}_{k\in\N}$. Then there exist two subsequences $\{x^{k_i}\}_{k\in\N}$ and $\{x^{k_j}\}_{k\in\N}$, such that $x^{k_i} \weak v_1^*$ and $x^{k_j} \weak v_2^*$. By Step-3 of the proof of Theorem \ref{t_con_ran}, we have $\cP_{\ransf\cQ} v_1^*
 =\cP_{\ransf\cQ} v_2^*$ and further $\cT v_1^* = \cT v_2^* $, i.e., all the weak sequential cluster points of  $\{x^k\}_{k\in\N}$ share the same projection onto $\ransf\cQ$.
 
Condition (i): it implies that  $v_1^* = \cT v_1^* = \cT v_2^* = v_2^*$. This shows that the weak sequential cluster point of  $\{x^k\}_{k\in\N}$ is unique, and lies in $\Fixsf\cT=\zersf\cA$. The weak convergence of  $\{x^k\}_{k\in\N}$  is obtained by \cite[Lemma 2.46]{plc_book}.
 
Condition (ii): Assumption \ref{assume_A}-(i) implies that $\grasf\cA$ is closed in $\cH^\text{weak} \times \cH^\text{strong} $. Take a subsequence $\{x^{k_i}\}_{k\in\N}$, such that $x^{k_i} \weak v^*$. Then $\cT x^{k_i} \weak v^*$, since $x^{k_i} -\cT x^{k_i} \rightarrow 0$. Considering
$\cA \cT x^{k_i} \owns \cQ (x^{k_i} -\cT x^{k_i} )$, combining with $\cT x^{k_i} \weak v^*$ and $x^{k_i} -\cT x^{k_i} \rightarrow 0$, it then follows from the closedness of $\grasf\cA$ that $v^\star \in \zersf\cA=\Fixsf\cT$, which leads to  Condition (i).

Condition (iii): Considering the subsequence $\{x^{k_i}\}_{k\in\N}$ in (ii), the demiclosedness of $\cI-\cT$ yields that $(\cI-\cT)v^* = 0$, i.e., $v^*\in\Fixsf\cT$, which also  leads to Condition (i).
\end{proof}

\vskip.2cm 
In Theorem \ref{t_con}, Conditions (i) and (ii) correspond to \cite[Theorem 2.9, Corollary 2.10]{bredies_ppa}, respectively. Condition (iii), as mentioned in \cite[Remark 2.11]{bredies_ppa}, is hard to control and examine. The key problem here is that $\cT$ is not necessarily nonexpansive, and thus the general demiclosedness of $\cI-\cT$ cannot be obtained by Browder's principle \cite[Theorem 4.27]{plc_book}. However, it can be tackled case-by-case, by exploiting particular structures of $\cA$ and $\cQ$, and additional assumptions, see, for instance, \cite[Lemma 3.4]{bredies_2017}. 

\section{Reduced form of  degenerate PPM} \label{sec_reduced}
\cite[Theorems 2.13 and 2.14]{bredies_ppa} proposed a reduced form of degenerate PPM, and analyzed its convergence properties. We shall now proceed with a concise recapitulation, further develop some results along this line and finally make connections to our presented results. 

This reduction approach is based on the following properties $\cQ$ that is parallel to Fact \ref{f_Q}.
\begin{fact} \label{f_Q_2}
Under Assumption \ref{assume_Q}, the following hold.

{\rm (i) \cite[Proposition 2.3]{bredies_ppa}}
There exists a bounded and surjective operator $\cC: \cH \mapsto \cH'$ for some real Hilbert space $\cH'$, such that 
$\cQ = \cC^\top \cC$. Moreover, $\cC^\top: \cH'\mapsto\cH$ is injective in $\cH'$.

{\rm (ii) \cite[Fact 2.25]{plc_book}} $\ransf \cQ =\ransf \cC^\top$, $\kersf \cQ=\kersf \cC$.

{\rm (iii) \cite[Proposition 3.30]{plc_book}} $\cP_{\ransf \cQ} = \cC^\top (\cC\cC^\top)^{-1} \cC$, $\cP_{\kersf \cQ} = \cI- \cC^\top (\cC\cC^\top)^{-1} \cC$.

{\rm (iv)} $\cP_{\ransf\cQ} = \tilde{\cC}^\top \tilde{\cC}$, where $\tilde{\cC} =(\cC\cC^\top)^{-\frac{1}{2}}\cC$ is a normalized version of $\cC$. 

{\rm (v)} $\cP_{\ransf \cQ} = \cC^\dagger \cC$, $\cP_{\kersf \cQ} = \cI-   \cC^\dagger \cC$, where $\cC^\dagger = \cC^\top (\cC\cC^\top)^{-1} $ denotes the pseudo-inverse of $\cC$, and $\ransf \cC^\top = \ransf \cC^\dagger$.
\end{fact}

It was shown in \cite{fxue_drs,fxue_rima,bredies_ppa} that the scheme \eqref{T} can be reduced to a `smaller' space---$\cH'$, by removing the redundancy caused by the degeneracy of $\cQ$. The existing results are summarized below.
\begin{proposition} \label{p_reduced}
Under Assumptions \ref{assume_Q} and \ref{assume_local} , the following hold.

{\rm (i) \cite[Eq.(8)]{fxue_drs}} Denoting  $c^k = \cC x^k \in \cH'$,  \eqref{T} can be reduced to
\be \label{T_another}
c^{k+1} = \cTtilde c^k, \quad \textrm{where\ } \cTtilde := 
\cC \big( \cA+\cC^\top \cC \big)^{-1} \cC^\top.
\ee

{\rm (ii) \cite[Theorem 2.13]{bredies_ppa} and \cite[Theorem 2.4]{fxue_drs}} $\cTtilde$ can further be simplified as a resolvent:
\be \label{T_3}
\cTtilde = \big( \cI + (\cC \circ \cA^{-1} \circ \cC^\top)^{-1} \big)^{-1} 
:= J_{\cC \triangleright \cA}.
\ee
\end{proposition}

Following similar procedures in previous sections, we present several further results regarding $\cTtilde:\cH' \mapsto\cH'$.
\begin{proposition} \label{p_T_another}
Given \eqref{T_another} and \eqref{T_3} under Assumptions \ref{assume_Q} and \ref{assume_local}, the following hold.

{\rm (i)} $\cTtilde = J_{\cC \triangleright \cA} = \big( \cI + (\cC \circ \cA^{-1} \circ \cC^\top)^{-1} \big)^{-1} $ is single-valued.

{\rm (ii)} $\cI =\big( \cI + (\cC \circ \cA^{-1} \circ \cC^\top)^{-1} \big)^{-1} + \big( \cI + \cC \circ \cA^{-1} \circ \cC^\top \big)^{-1}   $.

{\rm (iii)} $\domsf \cTtilde = \ransf \big(\cI+  \cC \circ \cA^{-1} \circ \cC^\top \big) $. 

{\rm (iv)} $\domsf\cTtilde = \cH'$ and thus, $\cTtilde$ is well-defined, 

\hskip.4cm {\rm (1)} if and only if  $\cC\infimal \cA: \cH' \mapsto 2^{\cH'}$ is maximally monotone;

\hskip.4cm {\rm (2)} if  $\cA: \cH \mapsto 2^{\cH}$ is maximally monotone and $0\in \sri (\ransf\cA -\ransf \cC^\top)$.

{\rm (v)} Assuming (iv-1) or (iv-2) is satisfied,  if $\zersf\cA \ne\varnothing$, $c^k \weak c^\star \in \cC (\zersf\cA)$.
\end{proposition}
\begin{proof}
(i)--(iv) similar to Proposition \ref{p_equality}.

(v) in view of Theorem \ref{t_full} and Corollary \ref{c_full}.

(vi) According to Lemma \ref{l_fix} and Theorem \ref{t_con_ran}, observing \eqref{T_3}, we obtain that the weak limit $c^\star \in \Fixsf \cTtilde =  \zersf  (\cC \circ \cA^{-1} \circ \cC^\top)^{-1} = (\cC \circ \cA^{-1} \circ \cC^\top) (0) =\cC \circ (\cA^{-1}  (0) ) = \cC (\zersf \cA) $.
\end{proof}

This reduction approach has connections to our presented results, as shown below.
\begin{proposition} \label{p_link}
Regarding $\cT$ \eqref{T_single} and $\cTtilde$ \eqref{T_another} or \eqref{T_3}, the following hold.

{\rm (i)} $\domsf \cTtilde = \cC (\domsf \cT)$.

{\rm (ii)} $\domsf \cTtilde = \cH' \Longleftrightarrow 
\domsf \cT = \cH$.

{\rm (iii)} $\ransf \big(\cI+  \cC \circ \cA^{-1} \circ \cC^\top \big) 
= \cC  (\sqrt{\cQ})^\dagger \ransf \big(\cI+   \sqrt{\cQ} \cA^{-1} \sqrt{\cQ}  \big) $. 

{\rm (iv)} $\cP_{\ransf\cQ} \circ \cA^{-1}|_{\ransf\cQ}$ is maximally monotone, if and only if   $\cC  \circ \cA^{-1}  \circ \cC^\top$ is maximally monotone.

{\rm (v)} Both weak limits are linked via $x_\rsf^\star\in \cP_{\ransf\cQ} (\zersf \cA) \Longleftrightarrow c^\star = \cC (\zersf\cA)$.
\end{proposition}
\begin{proof}
(i) Observing \eqref{T_another}, we develop that 
$\cC (\domsf \cT) = \cC  \{x \in\cH: \cQ x\in \ransf(\cA+\cQ) \}
= \cC  \{x \in\cH: \cC^\top \cC x\in \ransf(\cA+\cQ) \}
=  \{c \in\cH': \cC^\top c \in \ransf(\cA+\cQ) \}
= \domsf \cTtilde$.

(ii) Combining (i) with Proposition \ref{p_equality}-(v), we have $\domsf \cTtilde = \ransf\cC = \cH' \Longleftrightarrow 
\domsf \cT \supseteq \ransf \cC^\top = \ransf\cQ
\Longleftrightarrow \domsf\cT = \domsf\cT +\kersf\cQ
\supseteq \ransf\cQ +\kersf\cQ = \cH 
\Longleftrightarrow \domsf \cT = \cH$.

(iii) Combine (i), Proposition \ref{p_T_another}-(iv), Proposition \ref{p_equality}-(viii) and note that $\cC (\kersf\cQ)
= \cC (\kersf\cC) =\{0\} $.

(iv) First, if $\cP_{\ransf\cQ}\cA^{-1}|_{\ransf\cQ}$ is maximally monotone, then for any given $(u,x)\in \ransf\cQ \times \ransf\cQ$, we have: 
\[
(u,x)\in \grasf (\cP_{\ransf\cQ} \cA^{-1}|_{\ransf\cQ}) \Longleftrightarrow 
\forall (v,y)\in \grasf (\cP_{\ransf\cQ} \cA^{-1}|_{\ransf\cQ}), \ 
\langle u-v|x-y \rangle \ge 0.
\]
Since $u,x,v,y \in \ransf\cQ = \ransf \cC^\top$, let $u=\cC^\top u'$ and $v=\cC^\top  v'$ for $u',v'\in \cH'$. Then, $x\in \cP_{\ransf\cQ}\cA^{-1}|_{\ransf\cQ} (u)=
 \cP_{\ransf\cQ}\cA^{-1} u=\cP_{\ransf\cQ}\cA^{-1} \cC^\top  u'$, and similarly, $y\in \cP_{\ransf\cQ}\cA^{-1}\cC^\top v'$. Then, we obtain
\[
(u',x)\in \grasf \big(\cP_{\ransf\cQ} \cA^{-1} \cC^\top  \big) \Longleftrightarrow 
\forall (v',y)\in \grasf\big(\cP_{\ransf\cQ} \cA^{-1} \cC^\top \big), \ 
\langle \cC^\top (u'-v')|x-y \rangle \ge 0.
\]
This is also $\langle u'-v' | \cC x - \cC y \rangle \ge 0$. Let $x'=\cC x$ and $y'=\cC y$ (since $\cC$ is surjective in $\cH'$), then it becomes
\[
(u',x')\in \grasf \big(\cC \circ  \cA^{-1} \circ \cC^\top \big)
\Longleftrightarrow
\forall (v',y')\in \grasf \big(\cC \circ  \cA^{-1} \circ \cC^\top \big), \ 
\langle u'-v' | x'-y' \rangle \ge 0.
\]
This shows that  $\cC \circ  \cA^{-1} \circ \cC^\top $ is maximally monotone. The converse statement can be proved similarly.

(v)  $x_\rsf^\star\in \cP_{\ransf\cQ} (\zersf \cA) \Longleftrightarrow c^\star = \cC x_\rsf^\star\in \cC\cP_{\ransf\cQ} (\zersf \cA) =  \cC\cP_{\ransf\cC^\top} (\zersf \cA) 
=  \cC  (\zersf \cA) $.
\end{proof}

We have a few further remarks by comparing our restricted monotonicity in $\cH$ with this dimensionality reduction to $\cH'$.

(1) The `smaller' space $\cH'$ is essentially isomorphic to $\ransf\cQ$---a proper subspace embedded in $\cH$.

(2) As discussed in Sect. \ref{sec_max}, it is problematic to define maximality of monotonicity of $\cA^{-1}|_{\ransf\cQ}$, due to the invalid inner product in $\cH$. That is why we propose  $\cP_{\ransf\cQ} \circ \cA^{-1}|_{\ransf\cQ}$, from which the maximality and restricted Minty's theorem can be properly applied. By contrast, there is no difficulty to define maximal monotonicity of  $\cC \circ  \cA^{-1} \circ \cC^\top $, since both the domain and range of  $\cC \circ  \cA^{-1} \circ \cC^\top $ have been restricted in a `smaller' space $\cH'$, where the inner product performs properly.

(3) When using  $\cP_{\ransf\cQ} \circ \cA^{-1}|_{\ransf\cQ} $, the underlying space remains $\cH$. We still have access to analyzing the iterative behaviours of \eqref{T} in $\kersf \cQ$. If we perform the dimensional reduction, it would be hopeless to retrieve any information of the kernel part from the reduced form \eqref{T_another}, and it is infeasible to investigate the single-valuedness and whole convergence in $\cH$.

\section{Applications to operator splitting algorithms} \label{sec_app}
The standard convex settings of ALM  and DRS have been studied under the degenerate PPM framework in \cite{fxue_drs,bredies_ppa,latafat_2017}. This section is devoted to  minimization of some non-convex functions, by showing that the convergence of ALM and DRS in range space of preconditioner still holds under restricted maximal monotonicity.  Since we are going to consider non-convex function, we will use  $\partial f$ to denote the limiting subdifferential (i.e., a set of limiting derivatives)  \cite[Definition 8.3-(b)]{rtr_book} of the non-convex function $f$ from now on.

\subsection{Augmented Lagrangian method (ALM)}
Let us now consider a constrained optimization problem:
\be \label{problem_alm}
\min_{u\in\cU} f(u),\qquad \text{s.t.\ } Bu = 0, 
\ee
where $B:\cU\mapsto \cY$ is a linear bounded operator and $f:\cU \mapsto \R\mcup \{+\infty\}$ is a proper function. The  ALM   reads as (see \cite[Eq.(7.2)]{taomin_2018} for example):
\be \label{alm_1}
 \left\lfloor \begin{array}{lll}
u^{k+1}   & \in & \Arg \min_u  f(u)  + \frac{\tau}{2}
\big\| Bu  -\frac{1}{\tau} p^{k} \big\|^2  ,\\
p^{k+1}  & =  & p^k - \tau  Bu^{k+1}.
\end{array} \right. 
\ee
Substituting $q^k:=\frac{1}{\tau} p^k$ into \eqref{alm_1}, the optimality condition of \eqref{alm_1} becomes:
\[
 \left\lfloor \begin{array}{lll}
0  & \in & \partial (\frac{1}{\tau} f) (u^{k+1} ) +B^\top 
(Bu^{k+1}  -q^{k} ) ,\\
0  & =  & q^{k+1}  - q^k +   Bu^{k+1},
\end{array} \right. 
\]
which fits \eqref{T} as
\be \label{alm_ppa}
\begin{bmatrix}
 0 \\ 0   \end{bmatrix} \in 
 \underbrace{  \begin{bmatrix}
\partial (\frac{1}{\tau} f)   &  -B^\top  \\ B &  0
    \end{bmatrix} }_\cA   \begin{bmatrix}
u^{k+1}  \\ q^{k+1}     \end{bmatrix}     
 + \underbrace{  \begin{bmatrix}
 0  &  0   \\    0  &   I   \end{bmatrix}  }_\cQ
 \begin{bmatrix}
 u^{k+1} - u^{k}  \\  q^{k+1} - q^{k}     \end{bmatrix}.
\ee

The properties of ALM in convex setting have been well understood. 
We now focus on the behaviours of ALM in  non-convex setting. 
 Note that in this case, the sequence $(u^k,q^k)$ generated by \eqref{alm_ppa} does not guarantee the optimality of $u^k$ in the $u$-step of  \eqref{alm_1}.   The next Corollary \ref{c_alm} presents step-by-step results, some of which have not been explored in \cite{fxue_coam,fxue_drs}.
\begin{corollary} \label{c_alm}
Given the ALM \eqref{alm_ppa}, the following hold.

{\rm (i)} The reduced form of \eqref{alm_ppa} is the iteration of $q^k\in\cY$ only: 
$q^{k+1} = \big( \cI+ (\cC \circ \cA^{-1} \circ \cC^\top)^{-1} \big)^{-1}
q^k$,  where $\cC = \begin{bmatrix}
 0 & I \end{bmatrix}$.
 
 {\rm (ii)} $\big( \cC \circ \cA^{-1} \circ \cC^\top \big)^{-1}
=  B\circ \big( \partial (\frac{1}{\tau} f) \big)^{-1} \circ  B^\top$.

{\rm (iii)} The reduced ALM can be rewritten as
\be \label{alm_reduced}
q^{k+1} = \Big( \cI+  B\circ \big( \partial (\frac{1}{\tau} f) \big)^{-1} \circ  B^\top  \Big)^{-1}
q^k.
\ee

 {\rm (iv)} \eqref{alm_reduced} is well-defined, if $B\circ \big( \partial (\frac{1}{\tau} f) \big)^{-1} \circ  B^\top: \cY \mapsto 2^\cY$ is maximally monotone.
\end{corollary}
\begin{proof}
(i) Proposition \ref{p_reduced}-(ii).

(ii) $\cQ$ can be decomposed as  $\cQ = \cC^\top \cC$ with $\cC = \begin{bmatrix} 0 &   I \end{bmatrix}$. Consider $b\in (\cC \circ \cA^{-1} \circ \cC^\top)^{-1} (a) \Longleftrightarrow a\in  (\cC \circ \cA^{-1} \circ \cC^\top) (b)
= \cA^{-1} (0,b) $.
Letting $ \begin{bmatrix}
c \\ d \end{bmatrix} \in   \cA^{-1} \begin{bmatrix} 0 \\   b \end{bmatrix} $, we then have $a=  d$ and
\[
\begin{bmatrix}
0 \\   b \end{bmatrix}
\in   \cA   \begin{bmatrix}
c \\ a \end{bmatrix}  =  \begin{bmatrix}
 \partial (\frac{1}{\tau} f)  &  -B^\top  \\ B &  0
    \end{bmatrix}   \begin{bmatrix}
c \\ a \end{bmatrix}  =\begin{bmatrix}
\partial (\frac{1}{\tau} f)  (c) -B^\top a \\ Bc \end{bmatrix}.
\]
This yields that $b=  Bc$ and $  B^\top a  \in \partial (\frac{1}{\tau} f) (c)$, i.e.,  $c \in \big( \partial (\frac{1}{\tau} f) \big)^{-1} ( B^\top a)$. Finally, we obtain $b \in \big( B\circ \big( \partial (\frac{1}{\tau} f) \big)^{-1} \circ  B^\top \big)  ( a) $, which completes the proof.

(iii) Combine (i) and (ii).

(iv) Theorem \ref{t_full}  or Proposition \ref{p_T_another}-(iv).
\end{proof}

Here we do not use conjugate of $f$ to express $\big( \partial (\frac{1}{\tau} f) \big)^{-1} $, because $f$ is not assumed to be convex, and this is also no longer a standard subdifferential of convex function.

Example \ref{eg_2} shows the ALM may fail to converge, even if $f$ is convex. This corresponds to the problem of \eqref{problem_alm}, with $f(u_1,u_2) = \max\{ e^{u_2}-u_1, 0\}$ and $B=\begin{bmatrix} 1 & 0 \end{bmatrix}$. This has no solution by simple inspection. This can also be implied by the reduced ALM, which is given as $q^{k+1} = (\cI+\cK)^{-1} q^k$, where $\cK = B\circ \big( \partial (\frac{1}{\tau} f) \big)^{-1} \circ  B^\top$. However, this is not a valid iterate, since $\domsf \cK = \{0\}$, and $\ransf(\cI+\cK) = (0,+\infty)$.

Now, we consider an example of the problem \eqref{problem_alm} with $B=\begin{bmatrix}
1 & 1& 1\end{bmatrix}$ and
\be \label{alm_eg_1}
f:\R^3\mapsto \R: x = (x_1,x_2,x_3) \mapsto
\frac{1}{2}(x_1+x_2+x_3)^2 
+(x_1-x_2)^4 - (x_1-x_2)^2+(x_1-x_3)^2
\ee
The results for this example are summarized below, and the proof is postponed to Appendix \ref{app_3}.
\begin{proposition} \label{p_alm_1}
Regarding the example \eqref{alm_eg_1}, if $\tau >0$, the following hold.

{\rm (i)} $f$ is proper and globally continuous, but non-convex.

{\rm (ii)} $\partial (\frac{1}{\tau} f)$ is non-monotone, and so is $\cA$. 

{\rm (iii)} $B\circ (\partial (\frac{1}{\tau} f))^{-1} \circ B^\top: \R\mapsto \R: z\mapsto \tau z$, i.e., $B\circ (\partial (\frac{1}{\tau} f))^{-1} \circ B^\top = \tau I$.

{\rm (iv)} $B\circ (\partial (\frac{1}{\tau} f))^{-1} \circ B^\top$ is  maximally monotone.

{\rm (v)} The reduced ALM is $q^{k+1} = \frac{1}{1+\tau} q^k$.

{\rm (vi)}   $q^{k} =(1+\tau)^{-k} \cdot q^0 \rightarrow 0$, as $k\rightarrow +\infty$. This achieves global linear convergence of rate $1/(1+\tau)$.
\end{proposition}

\vskip.2cm
The next example of \eqref{problem_alm} is with  $B=\begin{bmatrix}
1 & 1& 1\end{bmatrix}$ and
\be \label{alm_eg_2}
f:\R^3\mapsto \R: x = (x_1,x_2,x_3) \mapsto
|x_1| +|x_2| +|x_3| - |x_1-x_2|.
\ee
The results are as follows.
\begin{proposition} \label{p_alm_2}
Regarding the example \eqref{alm_eg_2}, if $\tau >0$, the following hold.

{\rm (i)} $f$ is proper and globally continuous, but non-convex.

{\rm (ii)} $\partial (\frac{1}{\tau} f)$ is non-monotone, and so is $\cA$. 

{\rm (iii)} $B\circ (\partial (\frac{1}{\tau} f))^{-1} \circ B^\top$ is given as
\[
B\circ \big( \partial (\frac{1}{\tau} f) \big)^{-1} \circ B^\top: 
\R\mapsto 2^\R:  a \mapsto 
\left\{ \begin{array}{ll}
(-\infty, 0], & \textrm{if\ }  a=-1/\tau; \\
\{ 0\}, & \textrm{if\ }  a\in (-1/\tau, 1/\tau); \\ {}
[0,+\infty), & \textrm{if\ }  a=1/\tau. 
\end{array}\right.
\]

{\rm (iv)} $B\circ (\partial (\frac{1}{\tau} f))^{-1} \circ B^\top$ is  maximally monotone.

{\rm (v)} The reduced ALM is $q^{k+1} = (1+\cN_C)^{-1}  q^k$, where $\cN_C$ is a normal cone (at some point) of the set $C = [-1/\tau, 1/\tau]$.

{\rm (vi)}  $\{q^k\}_{k\in\N} $ follows essentially an iterative projection process onto the set $C$. This is  a firmly nonexpansive projection mapping: 
\[
q^{k+1} = \cP_{C} (q^k) = \left\{ \begin{array}{ll}
-1, & \textrm{if\ } q^k< -1/\tau; \\
q^k, & \textrm{if\ }  q^k\in [-1/\tau, 1/\tau]; \\ {}
1, & \textrm{if\ }  q^k > 1/\tau. 
\end{array}\right.
\]
\end{proposition}

The proof can be found in Appendix \ref{app_4}.
Last, it should be emphasized that the well-definedness and convergence of the dual variable $\{q^k\}_{k\in\N}$ in all above examples implies  the convergence of $\{Bu^k\}_{k\in\N}$ (to 0), though $\{ u^k\}_{k\in\N}$ may not be uniquely determined by the primal step of ALM \eqref{alm_1}. In addition, if $q^k\rightarrow 0$ as $k\rightarrow +\infty$, it can be concluded that the constraint of $Bu=0$ does not affect the minimization of $f$, i.e., the optimal value of the problem \eqref{problem_alm} exactly equals to $\min f$ (without any constraints).

\subsection{Douglas--Rachford splitting (DRS)}
The whole convergence of the DRS iterations \eqref{drs} in terms of  $\{(u^k,w^k,z^k)\}_{k\in\N}$ has been shown by an {\it ad hoc} analysis of \cite{svaiter} and the recent systematic study of \cite{bredies_ppa}. We now generalize the DRS to the non-convex case of $f+g$. Note again that in this case, the sequence $(u^k,w^k,z^k)$ generated by \eqref{T} with the choices of $\cA$ and $\cQ$ given as in \eqref{drs_ppa} fulfills the condition of \eqref{p1}, but does not necessarily satisfy the optimality of the $u$ and $w$-steps of  \eqref{drs}. 

\begin{corollary} \label{c_drs}
Given the DRS iterations \eqref{drs} and the corresponding $\cA$ and $\cQ$ as \eqref{drs_ppa},  the following hold.

{\rm (i)} The reduced form of \eqref{drs} is the iteration of $z^k\in\cU$ only: 
$z^{k+1} = \big( \cI+ (\cC \circ \cA^{-1} \circ \cC^\top)^{-1} \big)^{-1}
z^k$,  where $\cC = \begin{bmatrix}
 0 & 0 & I \end{bmatrix}$.
 
 {\rm (ii)} $\big( \cC \circ \cA^{-1} \circ \cC^\top \big)^{-1}
=  B\circ \cK^{-1}    B^\top$,  where 
$\cK =  \begin{bmatrix}
\tau \partial g & I \\  -I & \tau \partial f \end{bmatrix}$  and 
$B = \begin{bmatrix} I & -I  \end{bmatrix}$.

{\rm (iii)} The reduced DRS can be rewritten as
\be \label{alm_reduced2}
z^{k+1} = \big( \cI+ B\circ \cK^{-1} \circ B^\top \big)^{-1}
z^k.
\ee

 {\rm (iv)} \eqref{alm_reduced2} is well-defined, if $B\circ \cK^{-1} \circ B^\top: \cU \mapsto 2^\cU$ is maximally monotone.
 \end{corollary}
\begin{proof}
(i) Proposition \ref{p_reduced}-(ii).

(ii) $\cA$ can be rewritten as $\cA =  \begin{bmatrix}
\cK & -B^\top \\ B & 0 \end{bmatrix}$, and $\cQ$ can be decomposed as $\cQ = \cC^\top \cC$. Consider $b\in (\cC \circ \cA^{-1} \circ \cC^\top)^{-1} (a) \Longleftrightarrow a\in  (\cC \circ \cA^{-1} \circ \cC^\top) (b) = \cC \cA^{-1} (0,0,b)$.
Letting $ \begin{bmatrix}
c \\ d \end{bmatrix} \in   \begin{bmatrix}
\cK & -B^\top \\ B & 0 \end{bmatrix}^{-1} \begin{bmatrix}
0 \\  b \end{bmatrix}$, then we have $a= d$ and
\[
\begin{bmatrix}
0 \\  b \end{bmatrix}
\in    \begin{bmatrix}
 \cK  &  -B^\top  \\ B &  0
    \end{bmatrix}   \begin{bmatrix}
c \\ a \end{bmatrix}  =\begin{bmatrix}
\cK c -B^\top a \\ Bc \end{bmatrix}.
\]
This yields that $b= Bc$, $ B^\top a  \in \cK c$,  and then 
$b \in \big(  B\circ \cK^{-1}  \circ   B^\top \big)  ( a) $, which completes the proof.

(iii) Combine (i) and (ii).

(iv) Theorem \ref{t_full}  or Proposition \ref{p_T_another}-(v).\end{proof}

\vskip.2cm
Now, we consider to minimize $\min_{u\in\R}  \frac{1}{2}u^2 -|u|$, which is naturally split as the convex and non-convex part: $f(u)=\frac{1}{2}u^2$ and $g(u)= -|u|$.  The following proposition shows that the reduced DRS converges under restricted maximal monotonicity, even if $\cT$ is not well-defined.
\begin{proposition} \label{p_drs}
Regarding the DRS scheme \eqref{drs} and \eqref{drs_ppa} with $f(u)=\frac{1}{2}u^2$ and $g(u)= -|u|$, let $\tau = 1$, the following hold.

{\rm (i)} $f+g$ is proper and globally continuous, but non-convex.

{\rm (ii)} $\cK$ defined as Corollary \ref{c_drs}-(ii) is non-monotone, and so is $\cA$. 

{\rm (iii)} $B\circ \cK^{-1} \circ B^\top: \R\mapsto \R: z\mapsto z$, i.e., $B\circ \cK^{-1} \circ B^\top = I$.

{\rm (iv)} $B\circ \cK^{-1} \circ B^\top$ is maximally monotone.

{\rm (v)} The reduced DRS becomes $z^{k+1} = \frac{1}{2} z^k$.

{\rm (vi)}   $z^{k} = 2^{-k} \cdot z^0 \rightarrow 0$, as $k\rightarrow +\infty$. This achieves global linear convergence of rate $1/2$.
\end{proposition}

The proof is postponed to Appendix \ref{app_5}. Furthermore, let us consider the whole form of DRS for this case, which is given as
\be  \label{drs_eg}
\left\lfloor \begin{array}{lll}
0 & = &  \partial (-|u|)(u^{k+1}) + u^{k+1} - z^k, \\
0 & = & 2 w^{k+1} - 2u^{k+1}+   z^{k} , \\
0 & =  & \bz^{k+1}  -\bz^k - \bw^{k+1} + u^{k+1}.
 \end{array} \right.
\ee 
If we use limiting subdifferential, the $u$-step of \eqref{drs_eg} becomes
\[
u^{k+1} =  \left\{ \begin{array}{ll}
z^k - 1, & \textrm{if\ }  z^k < -1; \\
\{z^k-1, z^k+1\}, & \textrm{if\ }  z^k\in [-1, 1]; \\ {}
z^k+1, & \textrm{if\ }  z^k >1.
\end{array}\right.
\]


This suggests that $u^{k+1}$ is not uniquely determined by $z^k$, if $z^k\in [-1,1]$. However, it is interesting to observe the shadow sequence of $\{z^k\}_{k\in\N}$ still converges, though this DRS and the corresponding $\cT$ is not well-defined.

In numerical test, we initialize $z^0=1$ and randomly choose $u^{k+1}$ as either $z^k-1$ or $z^k+1$, when $z^k \in [-1,1]$ during each iterate. The sequences of  $\{(u^k,w^k,z^k)\}_{k\in\N}$ are shown in Fig. \ref{fig}-(1) and (2). It is easy to see that $z^k \rightarrow 0$ very rapidly, and both $u^k$ and $w^k$ jump between $+1$ and $-1$ after a few steps. Fortunately, both of $+1$ and $-1$ are global minimizers of $f+g$. Therefore, the objective value of $f+g$ keeps decreasing to its minimum of -1/2, as shown in Fig. \ref{fig}-(3).

\begin{figure} [H] 
\centering
\hspace*{-.2cm}
\scalebox{0.7} {
\begin{tikzpicture}[scale=0.9]

\draw node[below] at(0, 8.4) { \large  (1) evolution of $\{(u^k,w^k)\}_{k\in\N}$ };
\draw node[below] at (7, 8.4) { \large  (2) convergence of  
$\{z^k\}_{k\in\N}$ };
\draw node[below] at (14, 8.4) { \large (3) objective value $(f+g) (u^k)$};

\node at (0, 5) {\includegraphics[width=6cm]{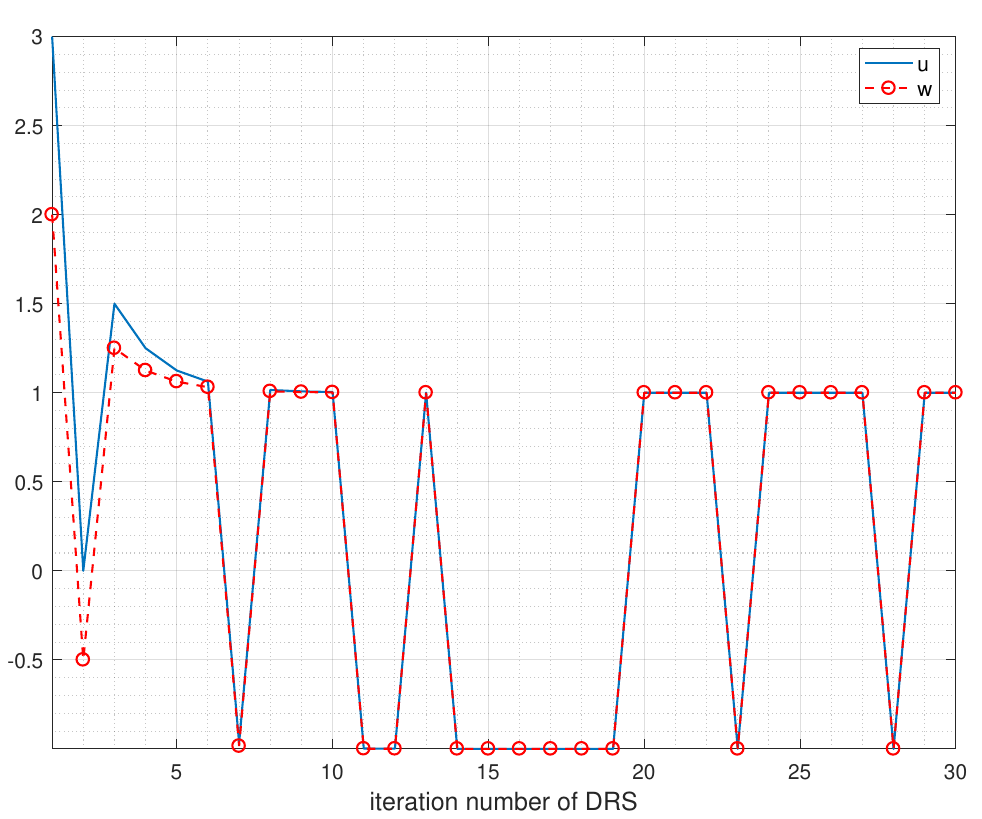}};
\node at (7, 5) {\includegraphics[width=6cm]{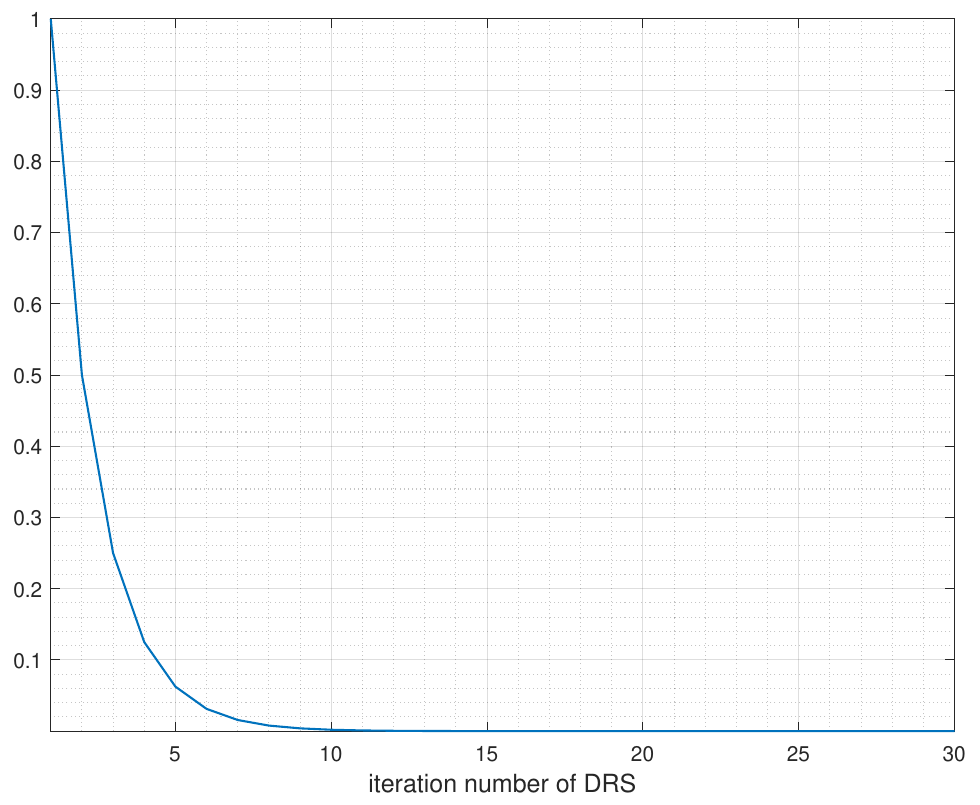}};
\node at (14, 5) {\includegraphics[width=6cm]{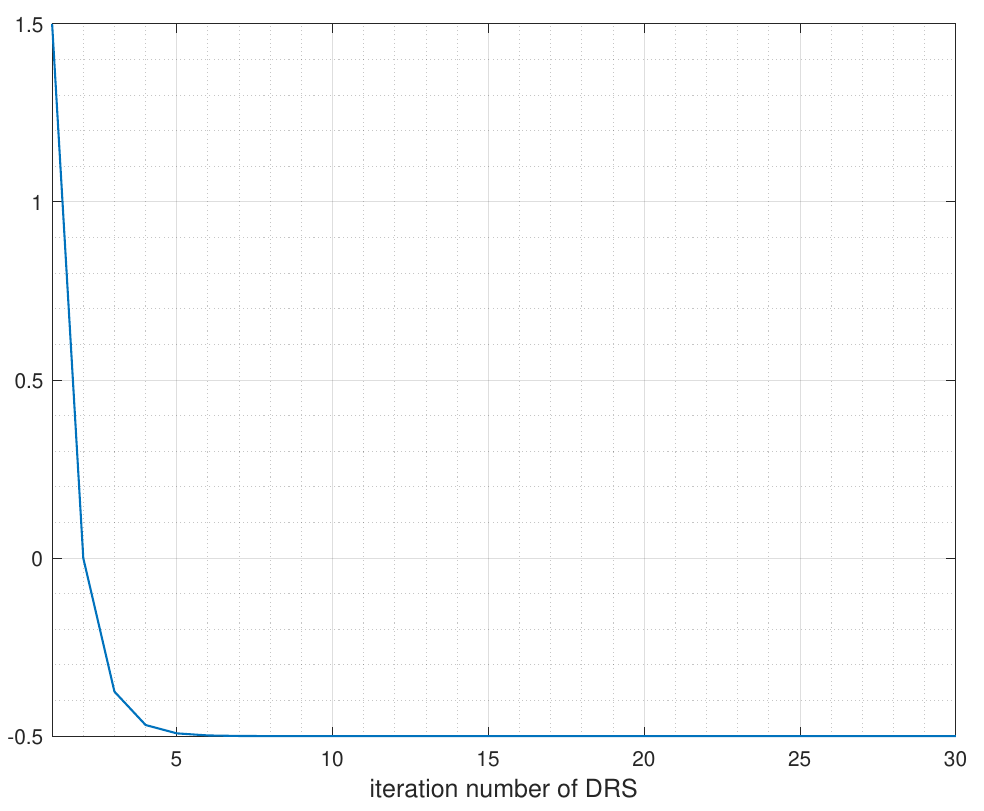}};

\end{tikzpicture} }
\vskip-.3cm
\caption{\small Evolutions during the DRS iterations. }
\label{fig}
\vskip-.08cm
\end{figure}

\section{Concluding remarks}
This work provided a fundamental understanding of degenerate preconditioner and its associated convergence issue. The applications of ALM and DRS have been extended to non-convex problems under restricted maximal monotonicity.  Our  results could be further exemplified with more scenarios, for example:
\begin{itemize}
\item Analyze the existence and uniqueness of solution to the general regularized least-squares problem: 
$\min_{x\in\R^n} f(x)+\frac{1}{2}\|Ax-b\|^2$,
particularly  generalized lasso (corresponding to $f=\|D\cdot\|_1$), and make close connections to \cite{fadili_cone}.

\item Proceed to study the iterative behaviours of ALM and DRS for no-convex problems satisfying restricted maximal monotonicity.
\end{itemize}

Some theoretical aspects are also of great interest, for instance,
\begin{itemize}
\item Explore the properties of $\cT$, if the range of $\cQ$ is not closed, as indicated in \cite{hhb_on_bredies}.

\item Study the behaviours of \eqref{T} if the inclusion problem \eqref{p1} is inconsistent, as proposed in \cite{hhb_drs}.

\item  Finding a simple interpretation of Assumption \ref{assume_A}-(ii) could also be interesting for practitioners. 
\end{itemize}

\section{Acknowledgements}
This work was supported by the National Natural Science Foundation of
 China under grant  12471300 and Independent Innovation Science Fundation of National University of Defense Technology (25-ZZCX-JDZ-14). We are gratefully indebted to  Mr. Brecht Evens (KU Leuven, Belgium) for inspiring comments and  bringing \cite[Theorem 3.4]{latafat_2017} to our attention in the early stages of this work.  We would also like to  thank Prof. Jianchao Bai (NWPU, China) and Mr. Shuchang Zhang (NUDT, China) for detailed discussions on DRS algorithm and \cite[Theorem 2.1]{fxue_drs}, and Dr. Emanuele Naldi (Universit\`{a} di Genova, Italy) for sharing his Ph.D. thesis \cite{naldi_thesis} and stimulating discussions.

\appendix
\section{An alternative proof of Theorem \ref{t_con_ran}}
\label{app_1}
The proof is presented below.
\begin{proof}
Substituting $x_\rsf = x_\rsf^k$ and $y_\rsf = x_\rsf^{k+1}$ in \eqref{rr3}, we have:
\be \label{abc}
\sqrt{\cQ} x_\rsf^k - \sqrt{\cQ} x_\rsf^{k+1} =
\Big( \cI+\big(\sqrt{\cQ}\cA^{-1}\sqrt{\cQ} \big) \big|_{\ransf\cQ} \Big)^{-1} \sqrt{\cQ} x_\rsf^k , 
\ee
where the single-valuedness of the right-hand side has been shown in Proposition \ref{p_equality}-(i).
Let  $q_\rsf^k := \sqrt{\cQ} x_\rsf^k$, $\cB := \big(\sqrt{\cQ}\cA^{-1}\sqrt{\cQ} \big) \big|_{\ransf\cQ}$, then, \eqref{abc} is rewritten as (by maximal monotonicity of $\cB:\ransf\cQ \mapsto 2^{\ransf\cQ}$ and celebrated Moreau's identity):
\[
 q_\rsf^{k+1} = q_\rsf^k - ( \cI+\cB )^{-1}  q_\rsf^k
= \big( \cI+\cB^{-1} \big)^{-1}  q_\rsf^k
\Longrightarrow
\cB^{-1} q_\rsf^{k+1} \owns q_\rsf^k-q_\rsf^{k+1}.
\]
It is known from the proof of Theorem \ref{t_con_ran} that $q_\rsf^k-q_\rsf^{k+1} \rightarrow 0$ and $ q_\rsf^{k} \weak q^\star$, as $k\rightarrow\infty$. Since $\cB^{-1}$ is also maximally monotone (by \cite[Proposition 20.22]{plc_book}), then $\grasf \cB^{-1}$ is sequentially closed in $(\ransf\cQ)^\text{weak}\times (\ransf\cQ)^\text{strong}$,  i.e., $(\ransf\cQ, \|\cdot\|_\cQ) \times (\ransf\cQ, \langle \cdot |\cdot \rangle_\cQ)$, by \cite[Proposition 20.38]{plc_book}. Thus, we obtain 
\[
0\in \cB^{-1} q^\star \Longrightarrow q^\star \in \cB (0)
=\sqrt{\cQ}\cA^{-1} (0) \Longrightarrow x_\rsf^\star \in \cP_{\ransf\cQ} \cA^{-1} (0) =\cP_{\ransf\cQ} (\zersf \cA).
\]
\end{proof}

\section{Proof of Proposition \ref{p_lip}}
\label{app_2}
Before proving this, we present several simple facts.
\begin{fact} \label{f_lip}
Under Assumptions \ref{assume_Q} and \ref{assume_local}, given $\cT$ in \eqref{T}, then $\forall x\in\ransf\cQ$, $\exists \lambda_{\inf}, \lambda_{\sup} >0$, such that 

{\rm (i)}  $\lambda_{\inf} \|x\|^2 \le \|x\|_\cQ^2 
\le \lambda_{\sup} \|x\|^2$;

{\rm (ii)} $\lambda_{\inf} \|x\| \le \|\cQ x\|  \le \lambda_{\sup} \|x\|$.
\end{fact}
\begin{proof}
Due to the boundedness and closedness of $\cQ$ as Assumption \ref{assume_Q}.
\end{proof}

Then, we prove Proposition \ref{p_lip}.
\begin{proof}
The proof is by the order of (i)$\Longrightarrow$(ii)$\Longrightarrow$(iii)$\Longrightarrow$(i).

(i)$\Longrightarrow$(ii): We have, $\forall x_1,x_2\in\ransf\cQ$, that
\begin{eqnarray}
&& \big\| (\cP_{\kersf\cQ} \circ \cT|_{\ransf\cQ}) x_1
- (\cP_{\kersf\cQ} \circ \cT|_{\ransf\cQ}) x_2  \big\|
\nonumber \\
&=& \big\| \cP_{\kersf\cQ}  (\cA+\cQ)^{-1} \cQ  x_1
- \cP_{\kersf\cQ} (\cA+\cQ)^{-1} \cQ  x_2  \big\|
\nonumber \\
&\le & \big\|  (\cA+\cQ)^{-1} \cQ  x_1
-  (\cA+\cQ)^{-1} \cQ  x_2  \big\|
\quad \textrm{by $\|\cP_{\kersf\cQ} \|\le 1$}
\nonumber \\
&\le & \xi \big\|  \cQ ( x_1-x_2)  \big\|
\quad \textrm{by (i)}
\nonumber \\
&\le & \xi \lambda_{\sup} \big\|   x_1-x_2  \big\|
\quad \textrm{by Fact \ref{f_lip}-(ii)}
\nonumber
\end{eqnarray}

(ii)$\Longrightarrow$(iii): We have, $\forall x_1,x_2\in\cH$, that
\begin{eqnarray}
&& \big\| \cT x_1 - \cT x_2  \big\|
\nonumber \\
&\le & \big\| \cP_{\ransf\cQ} \cT  x_{1,\rsf} -
 \cP_{\ransf\cQ} \cT  x_{2,\rsf}  \big\|
+ \big\| \cP_{\kersf\cQ} \cT  x_{1,\rsf} -
 \cP_{\kersf\cQ} \cT  x_{2,\rsf}  \big\|
\nonumber \\
&\le & \big\| x_{1,\rsf} - x_{2,\rsf}  \big\|
 + \big\| (\cP_{\kersf\cQ} \circ \cT|_{\ransf\cQ})  x_{1,\rsf} -
(\cP_{\kersf\cQ} \circ \cT|_{\ransf\cQ})  x_{2,\rsf}  \big\|
\quad \textrm{by Fact \ref{f_fne}}
\nonumber \\
&\le & (1+\eta) \big\| x_{1,\rsf} - x_{2,\rsf}  \big\|
\quad \textrm{by (ii)}
\nonumber \\
\nonumber \\
&\le & \frac{1+\eta}{\sqrt{\lambda_{\inf}}}
 \big\| x_{1} - x_{2}  \big\|_\cQ
\quad \textrm{by Fact \ref{f_lip}-(i)}
\nonumber
\end{eqnarray}

(iii)$\Longrightarrow$(i): We have, $\forall x_1,x_2\in\ransf\cQ$, that
\begin{eqnarray}
&& \big\| ((\cA+\cQ)^{-1}|_{\ransf\cQ}) x_1 - 
( (\cA+\cQ)^{-1}|_{\ransf\cQ}) x_2  \big\|
\nonumber \\
&\le &\big\| (\cA+\cQ)^{-1} \cQ y_1 -  
 (\cA+\cQ)^{-1} \cQ y_2  \big\|
\quad \textrm{since $\exists y_1,y_2$, s.t. $x_1=\cQ y_1$,
 $x_2=\cQ y_2$}
\nonumber \\
&\le & \xi \big\| y_{1} - y_{2}  \big\|_\cQ
 \quad \textrm{by (iii)}
\nonumber \\
&\le &  \frac{\xi \sqrt{\lambda_{\sup}}} 
{\lambda_{\inf} } 
 \big\| x_{1} - x_{2}  \big\|
\quad \textrm{by $x_1=\cQ y_1$, $x_2=\cQ y_2$ and
Fact \ref{f_lip}}
\nonumber
\end{eqnarray}
\end{proof}

\section{Proof of Proposition \ref{p_alm_1}}
\label{app_3}
The proof is presented below.
\begin{proof}
(i) Clear.

(ii) The limiting subdifferential of $\frac{1}{\tau} f$ is computed as
\[
\big(\partial (\frac{1}{\tau} f) \big) (x)
= \big(\nabla (\frac{1}{\tau} f) \big) (x) 
= \frac{1}{\tau} \begin{bmatrix}
(x_1+x_2+x_3) + 4(x_1-x_2)^3 - 2(x_1-x_2)+2(x_1-x_3) \\
(x_1+x_2+x_3) - 4(x_1-x_2)^3 + 2(x_1-x_2) \\
(x_1+x_2+x_3) -2(x_1-x_3) \\
\end{bmatrix}.
\]
Taking $x=(0.5,0,0)$ and $y=(0,0.5,0)$, we have $u=\nabla f(x)=(1,1,-0.5)$,
$v=\nabla f(y)=(1,0,0.5)$, and thus, $\langle x-y | u-v\rangle = -0.5<0$.

(iii) To see this, let us now first compute  $B\circ  (\partial f)^{-1} \circ B^\top$, and then take the factor of $\frac{1}{\tau}$ back into account.
For the input $a\in\R$ of the operator $B\circ (\partial f)^{-1} \circ B^\top$, let $b=(b_1,b_2,b_3)\in (\partial f)^{-1} ( B^\top a) =  (\partial f)^{-1} ( a,a,a)$, then the output $z \in (B\circ (\partial f)^{-1} \circ B^\top) (a)$ becomes $z=b_1+b_2+b_3$. We now evaluate $(b_1,b_2,b_3)\in  (\partial f)^{-1} ( a,a,a)$, which is equivalent to $\nabla f (b_1,b_2,b_3) = (a,a,a)$, i.e., 
\[
\left\{ \begin{array}{lll}
a & = & (b_1+b_2+b_3) + 4(b_1-b_2)^3 - 2(b_1-b_2)+2(b_1-b_3), \\
a & = &(b_1+b_2+b_3) - 4(b_1-b_2)^3 + 2(b_1-b_2) , \\
a & = & (b_1+b_2+b_3) -2(b_1-b_3). 
\end{array}\right.
\]
Adding the first two gives $a=  (b_1+b_2+b_3) +(b_1-b_3)$, which, combining with the last one, further yields $b_1=b_3$. This in turn implies that $ b_1+b_2+b_3 = a$. Finally, we have $B\circ (\partial f)^{-1} \circ B^\top: a \mapsto z =b_1+b_2+b_3= a$, which is identity. Taking the factor of $\frac{1}{\tau}$ into account completes the proof.

(iv) \cite[Corollary 20.28]{plc_book}.

(v)  Corollary \ref{c_alm}-(iii).

(vi) is clear from (v).
\end{proof}

\section{Proof of Proposition \ref{p_alm_2}}
\label{app_4}
The proof is presented below.
\begin{proof}
(i) Clear.

(ii) Taking $x=(1,0,0)$ and $y=(0,1,0)$, we have that
 $ \partial f(x)= \{0\}\times [0,2] \times [-1,1]$,
and  $ \partial f(y)= [0,2] \times  \{0\}  \times [-1,1]$. Choosing
$u =\frac{1}{\tau} (0,1,0) \in \partial  (\frac{1}{\tau} f) (x)$ and $v= \frac{1}{\tau} (1,0,0)\in\partial (\frac{1}{\tau} f) (y)$ yields that $\langle x-y | u-v\rangle = -2/\tau < 0$.

(iii) Still, let us first compute  $B\circ  (\partial f)^{-1} \circ B^\top$, and then take the factor of $\frac{1}{\tau}$ back into account.
For the input $a\in\R$ of the operator $B\circ (\partial f)^{-1} \circ B^\top$, let $b=(b_1,b_2,b_3)\in (\partial f)^{-1} ( B^\top a) =  (\partial f)^{-1} ( a,a,a)$, then the output $z \in (B\circ (\partial f)^{-1} \circ B^\top) (a)$ becomes $z=b_1+b_2+b_3$. We now evaluate $(b_1,b_2,b_3)\in  (\partial f)^{-1} ( a,a,a)$, which is equivalent to $(\partial f) (b_1,b_2,b_3) \owns (a,a,a)$. This inclusion implies that there exist subgradients  $v_1\in\partial |b_1|$, $v_2 \in\partial |b_2|$, $v_3\in\partial |b_3|$, and $w\in\partial |b_1-b_2|$, such that $v_1-w = a$, $v_2+w = a$ and $v_3 = a$. 
Observing the third component, we have that $a\in [-1, 1]$. Let us discuss it case-by-case.

Case-I: If $a=-1$, then $b_3 \le 0$. Suppose $b_1> b_2$, then $w=1$. This means that $v_1 -1 =a =-1$ and $v_2 +1 =a =-1$, i.e., $v_1 = 0$ and $v_2 = -2$. The latter is impossible. On the contrary,  suppose $b_1< b_2$, then $w=-1$. This means that $v_1 +1 =a =-1$ and $v_2-1 =a =-1$, i.e., $v_1 = -2$ and $v_2 = 0$. The former is impossible. The only possibility is that $b_1=b_2$. Further considering the sign of $b_1$ and $b_2$: If $b_1 =b_2 >0$, then $v_1=v_2=1$, and $w\in [-1,1]$. By $a=v_1-w=v_2+w$, we have that $a\in [0,2]$, which contradicts to $a=-1$.  If $b_1 =b_2 <0$, then $v_1=v_2=-1$, and $w\in [-1,1]$. By $a=v_1-w=v_2+w$, we have that $a\in [-2,0]$, which allows $a=-1$.  If $b_1 =b_2 =0$, then $v_1=v_2\in [-1,1]$, and $w\in [-1,1]$. By $a=v_1-w=v_2+w$, we have that $a\in [-2,2]$, which allows $a=-1$. Thus, we have $b_1=b_2\le 0$. Finally, $a=-1$ implies that $b_1=b_2 \le 0$ and $b_3\le 0$.

Case-II: If $a=1$, then $b_3 \ge 0$. Suppose $b_1> b_2$, then $w=1$. This means that $v_1 -1 =a =1$ and $v_2 +1 =a =1$, i.e., $v_1 = 2$ and $v_2 = 0$. The former is impossible. On the contrary,  suppose $b_1< b_2$, then $w=-1$. This means that $v_1 +1 =a =1$ and $v_2-1 =a =1$, i.e., $v_1 = 0$ and $v_2 = 2$. The latter is impossible. The only possibility is that $b_1=b_2$. Further considering the sign of $b_1$ and $b_2$: If $b_1 =b_2 >0$, then $v_1=v_2=1$, and $w\in [-1,1]$. By $a=v_1-w=v_2+w$, we have that $a\in [0,2]$, which allows $a=-1$.  If $b_1 =b_2 <0$, then $v_1=v_2=-1$, and $w\in [-1,1]$. By $a=v_1-w=v_2+w$, we have that $a\in [-2,0]$, which contradicts to $a=1$.  If $b_1 =b_2 =0$, then $v_1=v_2\in [-1,1]$, and $w\in [-1,1]$. By $a=v_1-w=v_2+w$, we have that $a\in [-2,2]$, which allows $a=-1$. Thus, we have $b_1=b_2\ge 0$. Finally, $a=1$ implies that $b_1=b_2 \ge 0$ and $b_3\ge 0$.

Case-III: If $a \in(-1,1)$, then $b_3 = 0$. Suppose $b_1> b_2$, then $w=1$. This means that $v_1 -1 = v_2+1 = a  \in (-1,1)$, i.e., $v_1 = v_2 +2$, which is impossible. On the contrary, suppose $b_1< b_2$, then $w=-1$. This means that $v_1 +1 = v_2-1 = a  \in (-1,1)$, i.e., $v_1 = v_2 -2$, which is also impossible. Finally, we have $b_1=b_2$. Regarding the signs of $b_1$ and $b_2$, if $b_1 =b_2 >0$, then  $v_1=v_2=1$, and $w\in [-1,1]$. By $a=v_1-w=v_2+w$, we have  $a=( v_1+v_2)/2=1$, which contradicts to $a\in (-1,1)$.  If $b_1 =b_2 <0$, then $v_1=v_2=-1$, and $w\in [-1,1]$. By $a=v_1-w=v_2+w$, we have $v_1+v_2 = -2 = 2a$, this implies that $a=-1$, also contradicting to $a\in (-1,1)$. Finally, the only possibility is $b_1=b_2=0$.
 
As summary of the above 3 cases, we have
\[
(\partial f)^{-1}: (a,a,a) \mapsto 
\left\{ \begin{array}{ll}
\{(b_1,b_2,b_3) : b_1=b_2 \le 0,\ b_3\le 0\}, & \textrm{if\ }  a=-1; \\
\{(b_1,b_2,b_3) : b_1=b_2=  b_3 = 0\}, & \textrm{if\ }  a\in (-1, 1); \\
\{(b_1,b_2,b_3) : b_1=b_2 \ge 0,\  b_3\ge 0\}, & \textrm{if\ }  a=1,
\end{array}\right.
\]
from which then follows 
\[
B\circ (\partial f)^{-1} \circ B^\top: \R\mapsto 2^\R:
 a \mapsto 
\left\{ \begin{array}{ll}
(-\infty, 0], & \textrm{if\ }  a=-1; \\
\{ 0\}, & \textrm{if\ }  a\in (-1, 1); \\ {}
[0,+\infty), & \textrm{if\ }  a=1. 
\end{array}\right.
\]
Taking the factor of $1/\tau$ back into account completes the proof.

(iv) Clear.

(v) Note that $B\circ (\partial (\frac{1}{\tau} f))^{-1} \circ B^\top$ is actually the subdifferential of an indicator function of the set $C = [-1/\tau,1/\tau]$
 defined on $\R$, i.e., $B\circ (\partial (\frac{1}{\tau} f))^{-1} \circ B^\top = \partial \iota_C = \cN_C$, which is a normal cone at some given point of $C$.
 
(vi) Rewrite (v) as $q^{k+1} + \cN_{C} (q^{k+1}) \owns q^k$, where 
\[
\cN_{C} (q) = \left\{ \begin{array}{ll}
[0, +\infty), & \textrm{if\ }  p =1/\tau; \\
0, & \textrm{if\ }  p\in (-1/\tau, 1/\tau); \\ {}
(-\infty, 0], & \textrm{if\ }  p =-1/\tau,
\end{array}\right.
\]
from which follows (vi). 
\end{proof}

\section{Proof of Proposition \ref{p_drs}}
\label{app_5}
The proof is presented below.
\begin{proof}
(i) Clear.

(ii) Taking $x=(1,0)$ and $y=(-1,0)$, we have $\partial g(x)= \{-1\}$, and $\partial g(y)= \{1\}$, and thus, $\cK x = (-1,-1)$, and $\cK y = (1,1)$, and  $\langle x-y | \cK x - \cK y\rangle = -4<0$.

(iii) For the input $a\in\R$ of the operator $B\circ \cK^{-1} \circ B^\top$, let $b=(b_1,b_2)\in \cK^{-1} ( B^\top a) =  \cK^{-1} ( a,-a)$, then the output $z \in (B\circ \cK^{-1} \circ B^\top) (a)$ becomes $z=b_1-b_2$. We now evaluate $(b_1,b_2)\in  \cK^{-1} ( a,-a)$, which is equivalent to $\cK  (b_1,b_2) = (a,-a)$. This implies the following system $\partial g(b_1) + b_2 \owns a$ and $-b_1 +\nabla f(b_2) = -a$. The second becomes $b_1-b_2=a$, which is exact the output $z$, i.e.,  $B\circ \cK^{-1} \circ B^\top: a \mapsto z= a$, which is identity.

(iv) \cite[Corollary 20.28]{plc_book}.

(v)  Corollary \ref{c_drs}-(iii).

(vi) is clear from (v).
\end{proof}

\bibliographystyle{siam}

\small{
\bibliography{refs}
}

\end{document}